\documentclass[11pt,reqno]{amsart}

\usepackage[text={155mm,225mm},centering]{geometry}
\usepackage[mathscr]{eucal}
\usepackage[all]{xy}
\usepackage{mathrsfs}
\usepackage{amsfonts,amsmath,amsthm,amssymb}
\usepackage{latexsym,tabularx,graphicx,pict2e}

\usepackage{hyperref}

\usepackage{pdfpages}

\newtheorem{theorem}{Theorem}[section]
\newtheorem{lemma}[theorem]{Lemma}

\newtheorem{corollary}[theorem]{Corollary}
\newtheorem{proposition}[theorem]{Proposition}
\newtheorem{proposition-definition}[theorem]{Proposition-Definition}
\newtheorem{example-proposition}[theorem]{Example-Proposition}

\newtheorem{question}[theorem]{Question}

\theoremstyle{definition}

\newtheorem{definition}[theorem]{Definition}
\newtheorem{notation}[theorem]{Notation}

\newtheorem{remark}[theorem]{Remark}

\newtheorem*{ack}{Acknowledgements}

\allowdisplaybreaks
\title{Commutativity of Quantization and Reduction for Quiver Representations}

\author{Hu Zhao}

\address{Department of Mathematics, Sichuan University, Chengdu, 
Sichuan Province 610064 P.R. China}

\email{zhaohu130@126.com}

\date{}

\begin{document}

\maketitle

\begin{abstract}
Given a finite quiver, its double may be viewed as its
non-commutative ``cotangent" space, and hence is a non-commutative
symplectic space. 
Crawley-Boevey, Etingof and Ginzburg constructed
the non-commutative reduction
of this space while Schedler constructed
its quantization. 
We show that the non-commutative
quantization and reduction commute with each other. Via
the quantum and classical trace maps, such a commutativity
induces the commutativity of the quantization
and reduction on the space of quiver representations.

\noindent{\bf Keywords:} necklace Lie algebra,
quiver variety, quantization, reduction, differential operator.

\noindent{\bf MSC 2020:} 16G20, 53D55, 81R60.
\end{abstract}

\setcounter{tocdepth}{1}
\tableofcontents

\section{Introduction}

In 2000 Kontsevich and Rosenberg proposed in \cite{KonRos2000} a heuristic
principle in the study of non-commutative geometry. It says that a non-commutative
geometric structure on a non-commutative space (in this article we mean an associative
algebra), if it exists, should induce its classical counterpart on its representation schemes.
This principle has achieved great success in the study of non-commutative Poisson geometry
(\cite{Cra2011,Van2008Double}) and non-commutative
symplectic geometry 
(\cite{BocLeB2000,CBEG2007,Gin2001}). 
The purpose of this paper
is to study the ``quantization commutes with reduction" problem
for quiver varieties, with the guidance of the Kontsevich-Rosenberg principle.

In 1980s Guillemin and Sternberg conjectured in \cite{GuiSte}
that for a symplectic manifold, 
its geometric quantization commutes with its reduction
and proved this conjecture for the case of compact K\"ahler manifolds.
In 1990s
Fedosov proved  in \cite{Fed1998}
that the deformation quantization commutes with reduction
for symplectic manifolds. 
At about the same time, Kontsevich, inspired by the mirror symmetry
from physics, initiated the study of ``noncommutative" symplectic
geometry (see \cite{Kontsevich,Kon94}).
The above problem of ``deformation quantization commutes with reduction" 
makes sense
for noncommutative symplectic manifolds, too, as mathematicians
have made much progress in this direction.
In what follows, we focus on the case of quiver algebras, which
already have
ample noncommutative symplectic/Poisson structures.

Let $Q$ be a finite quiver. Let $\overline Q$ be the double of $Q$; that is,
the quiver obtained from $Q$ by adding a new edge for each edge of $Q$ but with direction
reversed. 
According to Kontsevich \cite{Kontsevich} 
(see also \cite{CBEG2007,Gin2001}), $\overline Q$ is the ``non-commutative" cotangent
space of $Q$, and hence is a ``non-commutative" symplectic space.
Denote by $\mathrm{Rep}(\overline Q, \mathbf d)$ 
and $\mathcal M_{\mathbf d}(Q)$ the space of $\mathbf d$-dimensional
quiver representations of $\overline Q$ and the 
associated quiver variety, where the former
is a symplectic space and the latter is a Hamiltonian reduction of 
the former. We have the following results:

\vspace{2mm}
\noindent (1) Holland in \cite{Hol1999} showed that
for $\mathrm{Rep}(\overline Q, \mathbf d)$ and $\mathcal M_{\mathbf d}(Q)$,
their quantizations commute with the reduction procedure in the sense
of Fedosov.

\vspace{2mm}
\noindent (2) Ginzburg in \cite{Gin2001}, and simultaneously, Bocklandt-Le Bruyn in \cite{BocLeB2000}, 
showed that there is a Lie algebra structure on $(\mathbb{K}\overline{Q})_{\natural}$, 
called the {\it necklace Lie algebra}, 
such that the canonical trace map 
$$\mathrm{Tr}: (\mathbb{K}\overline Q)_\natural \to \mathbb{K}[\mathrm{Rep}(\overline Q,\mathbf{d})],\quad
\overline a\mapsto \{\rho\mapsto\mathrm{trace}(\rho(a))\}$$
is a map of Lie algebras, where the Lie bracket on the latter is the Poisson bracket. 

\vspace{2mm}
\noindent (3) Later Schedler constructed in \cite{Sch2005} 
a quantization of $(\mathbb{K}\overline{Q})_{\natural}$; 
such a quantization, denoted by $\mathbf N(Q)_\hbar$,
under the {\it quantum} trace map, is mapped to the differential operators 
on $\mathrm{Rep}(\overline Q,\mathbf d)$,
and hence gives a quantization of the latter. 

\vspace{2mm}
\noindent (4) Crawley-Boevey, Etingof and Ginzburg 
showed in 
\cite{CBEG2007} that $\mathbb{K}\overline Q$ in fact has a {\it bi-symplectic} structure,
which naturally induces the symplectic structure on $\mathrm{Rep}(\overline Q,\mathbf d)$,
and thus
makes the ``non-commutative symplectic structure" more precise. In loc cit they
also introduced the procedure of non-commutative Hamiltonian reduction for 
bi-symplectic spaces; for the quiver algebra $\mathbb{K}\overline{Q}$, the non-commutative 
reduction is nothing but a preprojective algebra, which denote by $\Pi Q$.

\vspace{2mm}
Based on these results,  we have the following 
two natural questions:

\begin{question}\label{questioninintro}
$(1)$ Do we have the non-commutative version of ``quantization commutes with reduction"
for bi-symplectic spaces?

$(2)$ Does the non-commutative ``quantization commutes with reduction" fit the Kontsevich-Rosenberg
principle, that is, does it induce the classical one on the
corresponding representation spaces?
\end{question}

In this paper, we try to give an answer to these two questions for quiver algebras. 
Our main theorem is:

\begin{theorem}\label{maintheorem}
Suppose $Q$ is a finite quiver, $\mathbf{d}$ is a dimension vector such that 
moment map $\mu$ is a flat morphism. Then:

$(1)$ \textup{(Noncommutative quantization commutes with reduction)}
There exists a non-commutative reduction 
$\mathcal{R}_{q}(\mathbf{N}(Q)_{\hbar},\widehat{\mathbf{w}})$ of $\mathbf{N}(Q)_{\hbar}$, 
which quantizes preprojective algebra $\Pi Q$. Moreover, the following diagram
\begin{equation}
\begin{split}
\xymatrixcolsep{4pc}
\xymatrix{
\mathbf{N}(Q)_{\hbar} \ar@{-->}[r] 
& \mathcal{R}_{q}(\mathbf{N}(Q)_{\hbar},\widehat{\mathbf{w}})\\
(\mathbb{K}\overline Q)_\natural \ar@{~>}[u] \ar@{-->}[r]&(\Pi Q)_\natural
\ar@{~>}[u]}
\end{split}
\end{equation}
commutes. 
		
$(2)$ \textup{(Existence of quantum trace)}
There is a quantum trace map $\mathrm{Tr^{q}}$, which maps 
$\mathcal{R}_{q}(\mathbf{N}(Q)_{\hbar},\widehat{\mathbf{w}})$ 
to the differential operators on
the quiver variety $\mathcal M_{\mathbf d}(Q)$, such that the following diagram
\begin{equation}\begin{split}
\xymatrixcolsep{4pc}
\xymatrix{
\mathbf{N}(Q)_{\hbar} \ar[r]^-{\mathrm{Tr}^{q}} \ar@{-->}[d]
& 	{\mathcal{D}_{\hbar}(\mathrm{Rep}(Q,\mathbf{d}))} \ar@{-->}[d]\\
\mathcal{R}_{q}(\mathbf{N}(Q)_{\hbar}, \widehat{\mathbf{w}}) \ar[r]^-{\mathrm{Tr}^{q}}& 
\displaystyle\frac{(\mathcal{D}_{\hbar}
(\mathrm{Rep}(Q,\mathbf{d})))^{\mathrm{GL}_{d}}}{( \mathcal{D}_{\hbar}
(\mathrm{Rep}(Q,\mathbf{d}))\ (\tau-\hbar \chi_{0}) (\mathfrak{gl}_{\mathbf{d}}))^{\mathrm{GL}_{d}}}		
}
\end{split}\end{equation}
commutes, where $\mathcal D_{\hbar}(-)$ are the differential operators
on the corresponding spaces.
\end{theorem}

In the above theorem, the curved arrow means quantization, 
the dotted arrow means the procedure of reduction, and the
``commutativity" of the diagram is understood as in Fedosov \cite[Theorem 3]{Fed1998}.
Thus combining the above theorem with the above cited works, we in fact get the following
commutative diagram
\begin{equation}\label{maincor}
\begin{split}
\xymatrixrowsep{0.8pc}
\xymatrixcolsep{1.2pc}
\xymatrix{
\mathbf{N}Q_{\hbar}\ar@{-->}[rd] \ar[rr]^-{\mathrm{Tr}^{q}}  && 
{\mathcal{D}_{\hbar}(\mathrm{Rep}(\overline Q,\mathbf{d}))} \ar@{-->}[rd] \\
&\mathcal{R}_{q}(\mathbf{N}(Q)_{\hbar}, \widehat{\mathbf{w}}) \ar[rr]^{\mathrm{Tr}^{q}}  
&& 
\displaystyle\frac{(\mathcal{D}_{\hbar}
(\mathrm{Rep}(Q,\mathbf{d})))^{\mathrm{GL}_{d}}}{( \mathcal{D}_{\hbar}
(\mathrm{Rep}(Q,\mathbf{d}))\ (\tau-\hbar \chi_{0}) (\mathfrak{gl}_{\mathbf{d}}))^{\mathrm{GL}_{d}}} \\
(\mathbb{K}\overline{Q})_{\natural} \ar@{~>}[uu] \ar@{-->}[rd] \ar[rr]^{\mathrm{Tr}}&& 
\mathbb{K}[\mathrm{Rep}(Q,\mathbf{d})] \ar@{~>}[uu] \ar@{-->}[rd] \\
&(\Pi Q)_{\natural} \ar@{~>}[uu] \ar[rr]^{\mathrm{Tr}} 
&& \mathbb{K}[\mathcal{M} _{\mathbf{d}}(Q)],\ar@{~>}[uu]
}
\end{split}
\end{equation}
This diagram exactly says that ``quantization commutes with reduction" fits the 
Kontsevich-Rosenberg principle,
and hence give an affirmative answer to Question \ref{questioninintro}.

The rest of the paper is devoted to the proof of the above theorem. It is organized as follows.
In \S\ref{sect:NCQR} we recall the bi-symplectic structure
introduced by Crawley-Boevey, Etingof and Ginzburg;
in \S\ref{sect:NCredandquant}
we show that that the non-commutative quantization commutes
with the non-commutative reduction for quiver algebras;
in \S\ref{sect:QR} 
we collect the results 
on the commutativity of quantization and reduction
for quiver representations, which is mainly due to Holland;
in \S\ref{sect:fromNC} we show that under the quantum and the classical trace maps,
the non-commutative version of ``quantization commutes with reduction"
induces the usual one on quiver representation spaces, and hence proves Theorem
\ref{maintheorem}.

\begin{ack}
The author would like to thank his advisor Professor
Xiaojun Chen as well as Professor Farkhod Eshmatov 
and Jieheng Zeng for help conversations; he also thanks Professor
Yongbin Ruan
for inviting him to visit IASM, Zhejiang University during the preparation
of the paper.
This work is partially supported by NSFC (Nos. 11890660 and 11890663).
\end{ack}

\section{Non-commutative bi-symplectic spaces}\label{sect:NCQR}

Crawley-Boevey, Etingof and Ginzburg introduced in \cite{CBEG2007}
a version of non-commutative symplectic spaces, which
they called {\it bi-symplectic spaces} and studied 
their non-commutative Hamiltonian reduction.
They also showed, as an important example, that the double of a quiver
is bi-symplectic.
In this section we briefly go over their results.

\subsection{Bi-symplectic spaces}
 
In this paper, $\mathbb{K}$ is an algebraically closed field of characteristic zero,
$R$ is the commutative 
semisimple $\mathbb{K}$-algebra $R=\oplus_{i=1} ^{n} \mathbb{K}e_{i}$ with
$e_ie_j=\delta_{ij} e_i$.  
Let $A$ be an $R$-algebra and $\mathbf{m}: A\otimes_{R} A \rightarrow A$ 
be the multiplication of $A$. 

\begin{definition}
Let $A$ be an $R$-algebra. The set of {\it non-commutative 
1-forms} of $A$ is 
$\Omega ^{1} _{R}A := \ker \mathbf{m}$.
\end{definition}

Equivalently, $\Omega^{1}_{R} A$ is the $A$-bimodule generated by 
$da$ for $a \in A$, 
subject to the following relations:
$$d(ab) = (da)b + a (db),\ \text{for any } a,\ b \in A.$$
Here, $d$ is considered as an $R$-linear map from $A$ to $\Omega^{1}_{R} A$.
 
\begin{definition}
An $R$-algebra $A$ is called {\it smooth} if it is finitely generated as an $R$-algebra and 
$\Omega^{1}_{R}A$ is projective as an $A$-bimodule.	
\end{definition}

From now on, all algebras are assumed to be smooth.

\begin{definition}
Let  $A$ be an $R$-algebra. The set of {\it non-commutative differential forms} 
of $A$,
denoted by $\Omega^{\bullet}_{R} A$, is 
the tensor algebra $T_{A} (\Omega^{1}_{R}A) = \oplus_{n \geq 0} T^{n}_{A}(\Omega_{R}^{1}A)$ 
equipped with differential $d : \Omega^{\bullet-1}_{R}A \rightarrow \Omega^{\bullet}_{R} A$, 
which is the extension of $d: A \rightarrow \Omega^{1}_{R}A$ by derivation and by setting
$d^2=0$.
\end{definition}

In practice, we write an $n$-form in the form $a_{0} da_{1} \cdots  da_{n}$ for $a_{i} \in A$.
   
\begin{definition}
Let $A$ be an $R$-algebra, and let $\Omega_{R}^{\bullet}A$ be 
the non-commutative differential forms of $A$. 
Then the {\it Karoubi-de Rham
complex} of $A$ over $R$ is the complex
\begin{displaymath}
\mathrm{DR}^{\bullet}_{R}A := \Omega _{R}^{\bullet}A/[\Omega_{R}^{\bullet}A , \Omega_{R}^{\bullet}A]_{s}
\end{displaymath}
with the differential induced from $\Omega _{R}^{\bullet}A$.
\end{definition}

Here $[-,-]_{s}$ means taking
the super commutators: for $\alpha \in \Omega^{i}_{R}A,\ \beta \in \Omega_{R}^{j} A$
$[\alpha, \beta]_{s} = \alpha \beta - (-1)^{ij} \beta \alpha$.
By definition, we have
$\mathrm{DR}^{0}_{R}A \cong A_{\natural},\ \mathrm{DR}^{1}_{R}A \cong \Omega^{1}_{R}A/[A,\Omega^{1}_{R}A]_{s}$.  

Let  $A$ be an $R$-algebra;
there exist two $A$-bimodule structures on $A \otimes A$: one is the outer 
and the other is the inner, given by $a(x \otimes y)b : = (ax) \otimes (yb)$
and $a*(x \otimes y)* b := (xb) \otimes (ay)$, respectively, for any $a,\ b,\ x,\ y \in A$.
 
\begin{definition}
Let $A$ be an $R$-algebra. The set of {\it double derivations} on 
$A$, denoted by $\mathbb D\mathrm{er}_R A$,
is the set of $R$-derivations from $A$ to the outer $A$-bimodule
$A \otimes A$.
\end{definition}

Due to the inner $A$-bimodule structure on 
$A \otimes A$,  $\mathbb{D}\mathrm{er}_{R}A$ is also an $A$-bimodule.
Next we recall the non-commutative version of contractions and Lie derivatives.

\begin{definition}
Let $A$ be an $R$-algebra and $\Theta \in \mathbb{D}\mathrm{er}_{R}A$.
The {\it contraction} of the non-commutative forms of $A$ 
with $\Theta$ is the $A$-linear map 
is
$i_{\Theta}:\, \Omega_{R}^{\bullet}A\, \rightarrow\, 
\Omega_{R}^{\bullet}A \otimes \Omega_{R}^{\bullet}A$
given by
\begin{displaymath}
\ d\alpha_{1}\cdots d\alpha_{n} \mapsto \sum_{k = 1} ^{n} 
(-1)^{k-1} \big(d\alpha_{1}\cdots d\alpha_{k-1}\Theta^{'}
(\alpha_{k}) \big) \otimes \big(\Theta^{''}(\alpha_{k}) d\alpha_{k+1}\cdots d\alpha_{n}  \big).
\end{displaymath} 
\end{definition}

Here for a 1-form $d\alpha \in \Omega_{R}^{1}A$, 
$i_{\Theta} (d\alpha) = \Theta(\alpha) = \Theta^{'}(\alpha) \otimes \Theta^{''}(\alpha)$.

\begin{definition}
Let $A$ be an $R$-algebra and let $\Theta$ be a double derivation on $A$. 
For an $n$-form $a_{0}d a_{1} \cdots d a_{n} \in \Omega^{n}_{R}A$, the 
{\it Lie derivative}
of $a_{0}d a_{1} \cdots d a_{n}$ with respect to $\Theta$ is given by
\begin{align*}
&L_{\Theta} (a_{0}d a_{1} \cdots d a_{n})\\
&= \Theta^{'}(a_{0}) \otimes \Theta^{''}(a_{0}) d a_{1}\cdots  d a_{n} \\
&+
\sum_{k=1}^{n} \big(a_{0}d a_{1}\cdots  d a_{k-1} d\Theta^{'}(a_{k}) \big) 
	\otimes \big( \Theta^{''}(a_{k}) d a_{k+1} \cdots  d a_{n} \big)\\
&+ \big(a_{0}d a_{1}\cdots  d a_{k-1} 
	\Theta^{'}(a_{k}) \big) \otimes \big( d\Theta^{''}(a_{k}) d a_{k+1} \cdots  d a_{n} \big).
\end{align*}
\end{definition}

Now for $\alpha \otimes \beta \in \Omega^{k}_{R}A \otimes \Omega^{l}_{R}A$, 
we write
\begin{displaymath}
	(\alpha \otimes \beta)^{\diamond}:=(-1)^{kl}\beta \alpha \in \Omega^{k+l}_{R}A.
\end{displaymath}

\begin{definition}
Let $A$ be an $R$-algebra and let $\Theta$ be a double derivation on $A$. 
The {\it reduced contraction} and the {\it reduced Lie derivative} with respect to $\Theta$ are
given by
\begin{displaymath}
\iota _{\Theta}:\, \Omega^{\bullet}_{R}A \, \rightarrow \, 
\Omega^{\bullet-1}_{R}A,\ \alpha \mapsto \iota_{\Theta} 
\alpha = (i_{\Theta} \alpha) ^{\diamond}
\end{displaymath}
and
	\begin{displaymath}
		\mathcal{L}_{\Theta}:\, \Omega^{\bullet}_{R}A \, \rightarrow \, \Omega^{\bullet}_{R}A,\ \alpha \mapsto \mathcal{L}_{\Theta} \alpha = (L_{\Theta} \alpha)^{\diamond}.
	\end{displaymath}
\end{definition}

\begin{proposition}[\cite{CBEG2007} Lemmas 2.8.6 and 2.8.8] 
\label{proposition: Cartam formula and conctraction }	
$(1)$ The Cartan formulas hold:
\begin{displaymath}
i_{\Theta}\circ d - d \circ i_{\Theta} = L_{\Theta}\quad\mbox{and}\quad
d \circ \iota_{\Theta} - \iota_{\Theta} \circ d = \mathcal{L}_{\Theta},
\;\text{for any $\Theta \in \mathbb{D}\mathrm{er}_{R}A.$}
\end{displaymath}

$(2)$ Suppose $A$ is an $R$-algebra and 
$\Theta \in \mathbb{D} \mathrm{er}_{R}A$. Then
\begin{enumerate}
\item[\textup{(i)}] for any $\omega \in \Omega^{n}_{R}A$, the map 
$\omega \mapsto \iota_{\Theta} \omega$ descends to a well-defined map 
$$\iota_{\Theta}:\, \mathrm{DR}^{n}_{R}A\, \rightarrow\, \Omega^{n-1}_{R}A;$$

\item[\textup{(ii)}] for a fixed $\omega \in \mathrm{DR}_{R}^{n}A$, 
there exsits a homomorphism of $A$-bimodules:
$$\iota \omega:\, \mathbb{D}\mathrm{er}_{R}A\, \rightarrow\, 
\Omega^{n-1}_{R}A,\ \Phi \mapsto \iota_{\Phi}\omega;$$

\item[\textup{(iii)}] for any $\omega \in \Omega_{R}^{n}A$, 
the following diagram commutes:
\begin{displaymath}
\xymatrix{
(\mathbb{D}\mathrm{er}_{R}A)_{\sharp} \ar[d]_-{({\iota(\omega)})_{\sharp}} 
\ar[rr]^-{\mathbf{m}_{\star}}&&\mathrm{Der}_{R}A \ar[d]^-{i: \theta \mapsto i_{\theta} \omega} \\
(\Omega_{R}^{n-1}A)_{\sharp} \ar[rr]&& \mathrm{DR}_{R} ^{n-1}A.	
}
\end{displaymath} 
\end{enumerate}
\end{proposition}

Recall that $R=\oplus _{i \in Q_{0}} \mathbb{K}e_{i}$, 
those idempotents give an important double derivation
\begin{displaymath}
\Delta:\, A\, \rightarrow\, A \otimes A,\ a \mapsto \sum_{i} ae_{i} 
\otimes e_{i} - e_{i} \otimes e_{i}a.
\end{displaymath}
Now, we recall a version of non-commutative symplectic structure introduced in \cite{CBEG2007}.

\begin{definition}\label{definition:bi-symplectic}
Let $A$ be an $R$-algebra. A 2-form
$\omega \in \mathrm{DR}^{2} _{R}A$ is called {\it bi-symplectic} if it satisfies
\begin{enumerate}
\item $\omega$ is closed, that is $d\omega = 0$;
		
\item $\iota \omega:\, \mathbb{D}\mathrm{er}_{R}A\, 
\rightarrow\, \Omega_{R} ^{1} A$ 
is a bijection of $A$-bimodules.
\end{enumerate}
An $R$-algebra $A$ equipped with a bi-symplectic form $\omega$ is 
called a {\it bi-symplectic space}.
\end{definition}

\subsection{Representation spaces and trace maps}\label{Sect:Repspaces}

Let $R$ be the semisimple $\mathbb{K}$-algebra 
$\oplus_{i \in I} \mathbb{K}e_{i}$, where $I=\{1,2,3,\cdots,n\}$ is a finite set and 
let $A$ be an $R$-algebra. For a $\mathbb{K}$-vector space 
$V=\oplus_{i \in I} V_{i}$, where $V_{i}$ is a subspace of $V$. $V$ is 
canonically a left $R$-module and 
thus $\mathrm{End}_{\mathbb{K}}(V)$ is an $R$-algebra.

\begin{definition}
Let $A$ be an $R$-algebra, and let $V=\oplus_{i \in I} V_{i}$ be a $\mathbb{K}$-vector space. 
The {\it representation space} $\mathrm{Rep}(A,V)$ of $A$ on $V$ is 
\begin{displaymath}
\mathrm{Rep}(A,V):= \mathrm{Hom}_{\mathrm{Alg}_{R}}(A, \mathrm{End}_{\mathbb{K}}(V)),
\end{displaymath}
where $\mathrm{Alg}_{R}$ is the category of $R$-algebras.
\end{definition}

$\mathrm{Rep}(A,V)$ is equipped with a $\mathrm{GL}(V)$-action by conjugation, 
which makes $\mathrm{Rep}(A,V)$ into a $\mathrm{GL}(V)$-variety. 

In practice, we usually consider $V= \oplus_{i \in I} \mathbb{K}^{d_{i}}$ for a dimension vector 
$\mathbf{d} \in \mathbb{N}^{I}$; in this case, we denote 
$\mathrm{Rep}(A,V)$
by $\mathrm{Rep}(A,\mathbf{d})$.
As an affine variety, the coordinate ring of $\mathrm{Rep}(A,\mathbf{d})$ is characterized
by the following:

\begin{proposition}[\cite{Van2008Double} Section 7.1]
For any $A \in \mathrm{Alg}_{R}$ and $\mathbf{d} \in \mathbb{N}^{I}$, 
the coordinate ring $A_{\mathbf{d}}$ of $\mathrm{Rep}(A,\mathbf{d})$ is generated by 
$\{(a)_{ij}|1\leq i,j \leq |\mathbf{d}|:= \sum_{i} d_{i}\}$, subject to following relations:
for all $a,\ b \in A$, $k\in \mathbb{K}$,
\begin{enumerate}
		\item $(e_{k})_{ij} =\delta_{\phi (i),k} \delta_{ij}$;
		
		\item $(ka)_{ij}= k(a)_{ij}, (a + b)_{ij} = (a)_{ij} + (b)_{ij} $;
		
		\item $(ab)_{ij} = \sum_{u} (a)_{iu}(b)_{uj}$.
\end{enumerate}
\end{proposition}
Here $\phi:\{1,2,\cdots,|\mathbf{d}| \} \rightarrow I$ is defined by the property
$$\phi(i) = k \text{ if and only if } d_{1}+d_{2}+\cdots +d_{k-1}+1 \leq i \leq d_{1}+\cdots +d_{k}.$$

Similarly,
for an $n$-form $\omega=a_{0}d a_{1} \cdots  d a_{n} \in \Omega^{n}_{R}A$,
\begin{displaymath}
(\omega)_{ij} = \sum_{k_{1},..k_{n-1}}(a_{1})_{i,k_{1}} d(a_{1})_{k_{1},k_{2}}\cdots  
d (a_{n})_{k_{n-1},j}.
\end{displaymath}
This formula can be rewritten as products of matrices:
\begin{displaymath}
	(\omega) = (a_{0}) d(a_{1}) \cdots  d(a_{n})
\end{displaymath}
where $d(a_{i})$ is the matrix with $(u,v)$-entity being $d (a_{i})_{uv}$.

For $\Theta \in \mathbb{D}\mathrm{er}_{R}(A)$, the induced vector field, 
unlike the case of differential forms, is given in the following form.
For $1 \leq i,j,u,v \leq |\mathbf{d}|$ and $a \in A$,
\begin{displaymath}
\Theta_{ij}(a_{uv}) := {\Theta^{'}(a)}_{uj} {\Theta(a)^{''}}_{iv}.
\end{displaymath}

\begin{definition}\label{definition:trace map}
Suppose $A$ is an $R$-algebra, and $\mathbf{d} \in \mathbb{N}^{I}$ is a dimension vector. 
The {\it (classical) trace map} 
$\mathrm{Tr}:\, A\, \rightarrow\, A_{\mathbf{d}}$ 
is defined to be
\begin{displaymath}
A \rightarrow  A_{\mathbf{d}},\ 
a \mapsto\{\rho\mapsto tr(\rho(a))\}.
\end{displaymath}
\end{definition}

Here $tr(-)$ means taking the trace of the matrices.
Similarly, one can define the trace map $\mathrm{Tr}$ for differential forms.

\begin{proposition}[\cite{CBEG2007} Theorem 6.4.3]
\label{proposition:bi-symplectic goes to classical symplectic}
Given an $R$-algebra $A$ with bi-symplectic form $\omega$,
$\mathrm{Rep}(A,V)$ is a symplectic manifold with symplectic 
form $\mathrm{Tr}(\omega)$.
\end{proposition}

\subsection{Quivers and quiver representations}
Now let $Q$ be a finite quiver, 
and let $\mathbb KQ$ be the path algebra of $Q$.
Then a representation of $\mathbb KQ$ 
has an alternative description
which is given as follows.

Let $Q_{0}$ and $Q_1$ the sets of vertices
and arrows respectively. For $a \in Q_{1}$, $s(a)$ and
$t(a)$ are the source and the target of $a$ respectively.
Then a representation of $Q$ consists of the following collection of data:
\begin{itemize}
\item to each vertex $i$, we assign a $\mathbb{K}$-vector space $V_{i}$
with $\dim V_{i}=d_i$;
\item to each arrow $a$, we assign a linear operator 
	$f_{a}:\, V_{s(a)}\, \rightarrow \, V_{t(a)}$.
\end{itemize}
In what follows, we 
also write $\mathrm{Rep}(\mathbb KQ,\mathbf d)$
as $\mathrm{Rep}(Q,\mathbf d)$.

Let $\overline{Q}$ be the double of $Q$, 
which is obtained from $Q$ by adding opposite arrow $a^{*}$ for each $a \in Q_{1}$.
Then we have:

\begin{proposition}[\cite{CBEG2007} Proposition 8.1.1.]
\label{proposition: cotangent bundle of CQ is bi-symplectic}
Let $\mathbb{K} \overline{Q}$ be the path algebra of doubled quiver $Q$. 
Then $\mathbb{K} \overline{Q}$ is smooth
and the 2-form $\omega = \sum_{a\in Q} da d{a}^{*}$ is bi-symplectic.
In particular, $\mathrm{Rep}(\overline Q, \mathbf d)$
endows an induced symplectic structure.
\end{proposition}

\section{Non-commutative reduction and quantization}\label{sect:NCredandquant}

In this section, we study the non-commutative Hamiltonian reduction
and quantization for a doubled quiver, viewed as a bi-symplectic space.
The main result is Theorem \ref{thm:NCQR}, which says that for such
a bi-symplectic space,
its non-commutative quantization commutes with
its non-commutative reduction.

\subsection{The non-commutative Hamilton operator}

Let $\mathfrak{g}$ be a Lie algebra and $(C^{\bullet},d)$ a cochain
complex of $\mathbb{K}$-vector spaces; let ${\mathrm{Com}_{\mathbb{K}} }$ be the category of cochain complex of $\mathbb{K}$-vector spaces.

A  $\mathfrak{g}$-equivariant structure on $(C^{\bullet},d)$ is a pair of linear maps
\begin{displaymath}
L : \mathfrak{g} \rightarrow \mathrm{Hom}_{\mathrm{Com}_{\mathbb{K}}  }(C^{\bullet}, C^{\bullet}),\ x \mapsto L_{x} 
\end{displaymath}
and
\begin{displaymath}
i: \mathfrak{g} \rightarrow \mathrm{Hom}_{\mathrm{Com}_{\mathbb{K}}  }(C^{\bullet }, C^{\bullet -1}),\ x \mapsto i_{x}
\end{displaymath}
which satisfies that,
for any $x,y \in \mathfrak{g}$,
$$[L_{x},L_{y}] = L_{[x,y]},\ L_{x} = d\circ i_{x}  + i_{x}\circ d,\ 
i_{x} \circ i_{y} + i_{y} \circ i_{x} = 0,\ 
[L_{x}, i_{y}] = i_{[x,y]}.
$$

\begin{definition}
Let $(C^{\bullet},d)$ be a $\mathfrak{g}$-equivariant complex. 
A linear map $H: C^{1} \rightarrow \mathfrak{g}$ is called a {\it Hamilton operator} 
if it satisfies the following conditions: for any $\alpha,\ \beta \in C^{1}$,
\begin{itemize}
	\item[$(1)$] $i_{H(\alpha)} \beta + i_{H(\beta)} \alpha = 0$;
	\item[$(2)$] $[H(\alpha),H(\beta)] = 
		H( i_{H(\alpha) }  \circ d\beta - i_{H(\beta)} 
		\circ d\alpha + d \circ i_{H(\alpha)} \beta)$.
\end{itemize}
\end{definition}

\begin{proposition}[\cite{CBEG2007} Proposition 4.3.5]
Let $(C^{\bullet},d)$ be a $\mathfrak{g}$-equivariant complex and 
let $H: C^{1} \rightarrow \mathfrak{g}$ be a Hamilton operator.
Then there exists a Lie bracket on $C^1$ 
induced from $H$ by the formula 
$$\{{x,y}\} := i_{H(dx)}dy,\ \text{for any }x,\ y \in C^{1}.$$
\end{proposition}

From Definition \ref{definition:bi-symplectic} (2), we have a bijection
$$(\iota \omega)_{\natural}:\, (\mathbb{D}\mathrm{er}_{R} A)_{\natural} 
\, \rightarrow\, (\Omega_{R} ^{1} A)_{\natural} = \mathrm{DR}^{1}_{R}A.$$
On the other hand, 
the multiplication map $m: A \otimes A \rightarrow A$
induces a map 
$$\mathbf{m}_{*}:\mathbb{D}\mathrm{er}_{R} A \rightarrow \mathrm{Der}_{R}A,\ 
\Theta \mapsto \mathbf{m}\circ \Theta.$$
Hence, a bi-symplectic structure gives a Hamilton operator; that is,
we have:

\begin{proposition}[\cite{CBEG2007} Theorem 7.2.3]\label{prop:Kontsevich}
Let $A$ be an $R$-algebra with a bi-symplectic form $\omega$. Then we have the
following.
\begin{enumerate}
\item There exists a Hamilton operator
\begin{displaymath}
H_{\omega}:\, \mathrm{DR}^{1}_{R}A\, \rightarrow\, \mathrm{Der}_{R} A
\end{displaymath}
which is the composition
\begin{displaymath}
\xymatrixcolsep{1.6cm}\xymatrix{
\mathrm{DR}^{1}_{R}A \ar[r]^-{ (\iota \omega)_{\natural}^{-1}} 
& (\mathbb{D}\mathrm{er}_RA)_{\natural} \ar[r]^-{(\mathbf{m}_{*})_{\natural}}
& \mathrm{Der}_{R}A.
}
\end{displaymath}
\item $H_{\omega}$ induces a Lie bracket on $A_{\natural}$, explicitly,
\begin{displaymath}
\{ a,b\}:= i_{H_{\omega}(da)}(db),\ \text{for any $a,b \in A_{\natural}$}.
\end{displaymath}
\end{enumerate}
\end{proposition} 

The Lie bracket in Proposition \ref{prop:Kontsevich} (2) was first found by 
Kontsevich, and coincides with the {\it $H_{0}$-Poisson structure} introduced by 
Crawley-Boevey in \cite{Cra2011}, whose definition is the following:

\begin{definition}[\cite{Cra2011} Definition 1.1]
Let $A$ be an $R$-algebra.
A linear map 
$$p:\, A_{\natural}\, \rightarrow\, \frac{\mathrm{Der}_{R}A}{\mathrm{Inn}_{R}A}$$
is called an {\it $H_{0}$-Poisson structure}
if for any $a,\ b \in A$ with  $\overline{a},\ \bar{b}\in A_{\natural}$ respectively, 
\begin{displaymath}
\{\bar{a}, \bar{b}\}_{p}:= \overline{p(a)(b)}
\end{displaymath}
is a Lie bracket on $A_{\natural}$.
\end{definition}

\subsection{Noncommutative Hamiltonian reduction}

Suppose $R$ is the 
semisimple $\mathbb{K}$-algebra $\oplus_{i=1} ^{n} \mathbb{K}e_{i}$.
Suppose 
$A$ is a bi-symplectic $R$-algebra with bi-symplectic form $\omega$. 
For any $a\in A$, by Definition \ref{definition:bi-symplectic},
there exists $\mathcal{H}_{a}\in\mathbb D\mathrm{er}_RA$ such that
$$\iota_{\mathcal{H}_a}\omega=da\in\Omega^1_R A.$$

\begin{definition}[\cite{CBEG2007} Section 4.1 and 	
\cite{Van2008Double} Definition 2.6.4]
Let $A$ be as above.
A {\it non-commutative moment map} is
an element in $\mathbf{w}\in\oplus_i e_i A e_i$
such that
$$\mathcal{H}_{\mathbf w}(a)=\Delta(a)\in A\otimes A,$$
for any $a\in A$.
\end{definition}

\begin{definition}\label{definition:non-commutative reduction}
Suppose $(A,\omega)$ is a bi-symplectic space over $R$
and
$\mathbf{w}$ is a non-commutative moment map. 
Then the {\it non-commutative Hamiltonian reduction} of $A$ with respect to $\mathbf{w}$ is 
\begin{displaymath}
\mathcal{R}(A,\mathbf{w}) := \frac{A}{A \mathbf{w} A},
\end{displaymath}
where $A \mathbf{w} A$ is the two-side ideal generated by $\mathbf{w}$.
\end{definition}

\begin{proposition}[\cite{CBEG2007} Proposition 4.4.3 and Theorem 7.2.3]
\label{proposition:noncommutatiuve reduction H^0 Poisson}
Given an $R$-algebra $A$ with bi-symplectic form $\omega$ 
and non-commutative moment map $\mathbf{w}$.
\begin{enumerate}
\item For any $f\in A_{\natural}$, 
$\theta_{f}:= H_{\omega}(df)$ preserves $A \mathbf{w} A$, 
and thus descends to a well-defined derivation
\begin{displaymath}
\overline{\theta_{f}} \in \mathrm{Der}_{R}(\mathcal{R}(A,\mathbf{w})).
\end{displaymath}

\item The Lie bracket on $A_{\natural}$ descends to 
a well-defined Lie bracket on $(\mathcal{R}(A,\mathbf{w}))_{\natural}$.
\end{enumerate}
\end{proposition}

The following proposition says
the non-commutative Hamiltonian reduction
introduced above fits the Kontsevich-Rosenberg
principle, which
justifies its definition.

\begin{proposition}[\cite{CBEG2007} Theorem 6.4.3]

Given an $R$-algebra $A$ with bi-symplectic form $\omega$
and non-commutative moment map $\mathbf{w}$,
and a $\mathbb{K}$-vector space $V$
of finite dimension. Then:
\begin{enumerate}
\item Let $\mu:=(\mathbf{w}) \in \mathbb{K}
[\mathrm{Rep}(A,V)] \otimes \mathrm{End}_{\mathbb{K}}(V),$ 
then $\mu$ is a moment map for symplectic space $\mathrm{Rep}(A,V)$;
		
\item $\mathrm{Rep}(\mathcal{R}(A,\mathbf{w}),V)$ 
is a subscheme of $\mathrm{Rep}(A,V)$ and 
$\mathrm{Rep}(\mathcal{R}(A,\mathbf{w}),V) = \mu^{-1}(0)$.
\end{enumerate}
\end{proposition}

\begin{proposition}[\cite{CBEG2007} Proposition 8.1.1.]
Let $\mathbb{K} \overline{Q}$ be the path algebra of doubled quiver $Q$. Then
the following element
\begin{displaymath}
	\mathbf{w}:= \sum_{a\in Q} (a a^{*} - a^{*} a)
\end{displaymath}
is a non-commutative moment map.
\end{proposition}

It is also direct to see that, if $\mathbf w$ is a non-commutative moment map,
then for any $\lambda\in R$, so is $\mathbf w-\lambda$.

\begin{corollary}
Suppose $Q$ is a finite quiver. Then for any $\lambda \in R$,
\begin{enumerate}
\item the following algebra,
$\Pi^{\lambda} Q :=\displaystyle \frac{\mathbb{K} \overline{Q}}{(\mathbf{w} - \lambda)}$,
called the deformed preprojective algebra,
is the non-commutative Hamiltonian reduction of $\mathbb{K} \overline{Q}$ at 
$\mathbf{w} - \lambda$. 
In particular, $\Pi^{\lambda} Q $ is equipped with an $H_{0}$-Poisson structure;

\item suppose $V$ is a representation of $Q$ with dimension vector 
$\mathbf{d} \in \mathbb{N} ^{Q_{0}}$, and 
$\sum_{i \in Q_{0}} d_{i} \lambda_{i} = 0$. Then
	$\mathrm{Rep}(\Pi^{\lambda} Q,V) = \mu^{-1}(\lambda)$.
\end{enumerate}
\end{corollary}

\begin{proof}
The proof is given in \cite[Propositions 6.8.1 and 7.11.1]{Van2008Double};
see also \cite[Theorem 6.4.3]{CBEG2007}.
\end{proof}

From now on, the preprojective algebra $\Pi^{0}Q$ will be denoted by $\Pi Q$.

In summary, suppose $R$ is a commutative semisimple $\mathbb{K}$-algebra 
$\oplus_{i=1} ^{n} \mathbb{K}e_{i}$, $A$ is an $R$-algebra equipped with 
a bi-symplectic form $\omega \in \mathrm{DR}^{2}_{R}A$ and non-commutative 
momemt map $\mathbf{w}$. 
One constructs the non-commutative reduction 
$\mathcal{R}(A,\mathbf{w})$ of $A$ at $\mathbf{w}$. 
This procedure is denoted by
\begin{displaymath}
	\xymatrix{
	A \ar@{-->}[r] & \mathcal{R}(A,\mathbf{w})}.
\end{displaymath}

\subsection{Quantization of the necklace Lie algebra}

In this subsection, we go over the quantization of  
the non-commutative Poisson structure on $\mathbb{K}\overline Q$,
which is due to Schedler \cite{Sch2005}; the DG algebras case
has recently been studied by Chen and Eshmatov in \cite{CheEsh2020}.

By Definition \ref{definition:non-commutative reduction} and 
Proposition \ref{proposition:noncommutatiuve reduction H^0 Poisson},
there exists a Lie bracket on 
$$(\mathbb{K}\overline{Q})_{\natural}=\frac{\mathbb{K}\overline{Q}}{[\mathbb{K}\overline{Q},\mathbb{K}\overline{Q}]}.$$
More precisely, it is given as follows:
for any $a\in Q$,
set $\{a,a^{*}\}=1$
and  $\{a^{*},a\}=-1$;
set also
$\{f, g\}= 0$ 
for any $f, g\in \overline{Q} \ with\ f \neq g^{*}$.
For cyclic paths, $a_{1}a_{2}\cdots a_{k},b_{1}b_{2}\cdots b_{l}\in (\mathbb{K}\overline{Q})_{\natural}$
with $a_{i},b_{j}\in \overline{Q}$, we
have
\begin{align*}
&\{a_{1}a_{2}\cdots a_{k},b_{1}b_{2}\cdots b_{l}\}\\
&=
\sum_{1\leqslant i \leqslant k,\, 1\leqslant j \leqslant l}
\{a_{i},b_{j}\}t(a_{i+1})a_{i+1}a_{i+2}\cdots a_{k}a_{1}
\cdots a_{i-1} b_{j+1}\cdots b_{l}b_{1}\cdots b_{j-1};
\end{align*}
we also set $\{b, x\} = \{x, b\} = 0$ for any $b\in R$.

\begin{definition}
Suppose $Q$ is a finite quiver, $(\mathbb{K}\overline{Q})_{\natural}$ 
equipped with above bracket is called the {\it necklace Lie algebra} of $Q$.
\end{definition}

We next recall Schedler's construction of the quantization of necklace Lie algebras.
Let us first recall some notations.

\begin{notation}
Supppose $Q$ is a finite quiver and 
$R$ is the semisimple algebra 
$\oplus_{i \in Q_{0}} \mathbb{K}e_{i}$.
We recall the following notations introduced by Schedler in \cite{Sch2005}.
\begin{enumerate}
\item Let $AH:= \overline{Q} \times \mathbb{N}$, which is
called the space of arrows with heights. 

\item Let $E_{\overline{Q},h}$ be the $\mathbb{K}$-vector space spanned by $AH$.

\item Let $LH:= (T_{R}E_{\overline{Q},h})_{\natural}$, which is called
the generalized cyclic path algebra with heights.

\item Let $SLH[\hbar]:= SLH \otimes \mathbb{K} [\hbar]$, which
is the symmetric algebra generated by $LH$. 
\end{enumerate}
\end{notation}

In what follows 
all symmetric products, not just on $SLH[\hbar]$, are denoted by $\&$.

Consider the $\mathbb{K} [\hbar]$-submodule $SLH^{'}$ spanned by elements of the form
\begin{equation}
\label{general form}
\begin{split}
(a_{1,1},h_{1,1})\cdots (a_{1,l_{1}},h_{1,l_{1}}) \& (a_{2,1},h_{2,1})\cdots (a_{2,l_{2}},h_{2,l_{2}})
\\ \& \cdots  \& 
(a_{k,1},h_{k,1})\cdots (a_{k,l_{k}},h_{k,l_{k}}) 
\&v_{1} \& v_{2} \&\cdots \& v_{m}.
\end{split}
\end{equation}
where the $h_{i,j}$ are all distinct, $a_{i,j}\in \overline{Q}$ and
$v_{i}\in Q_{0}.$
Let $\tilde{A}$ be the quotient of
$SLH^{'}$ where two elements in $SLH^{'}$ are identified
if and only if
the order of heights in each element are preserved 
when we exchange each height in corresponding places.
 
Next, consider the $\mathbb{K}[h]$-submodule $\tilde{B}$ of $\tilde{A}$ generated by the following forms:
\begin{itemize}\label{quiver skein relations}
	\item $X-X^{'}_{i,j,i^{'},j^{'}} -X^{''}_{i,j,i^{'},j^{'}}$, \\[2mm]
	where $i \neq i^{'},h_{i,j}<h_{i^{'},j^{'}},\nexists (i^{''},j^{''})\ with\ h_{i,j} < h_{i^{''},j^{''}} < h_{i^{'},j^{'}}$;\vspace{2mm}
	\item $X-X^{'}_{i,j,i,j^{'}} - \hbar X^{''}_{i,j,i,j^{'}}$, \\[2mm]
	where $h_{i,j}<h_{i,j^{'}},\nexists (i^{''},j^{''})\ with\ h_{i,j} < h_{i^{''},j^{''}} < h_{i^{'},j^{'}}$
\end{itemize}
In the above, $X^{'}$ and $X^{''}$ are defined as follows:\\[2mm]
$X^{'}_{i,j,i^{'},j^{'}}$ is same as $X$ but with heights $h_{i,j}$ and $h_{i^{'},j^{'}}$ interchanged;
$X^{''}_{i,j,i^{'},j^{'}}$ replaces the components $(a_{i,1},h_{i,1})\cdots (a_{i,l_{i},h_{i,l_{i}}})\ and\ (a_{i^{'},1},h_{i^{'},1})\cdots (a_{i^{'},l_{i^{'}},h_{i^{'},l_{i^{'}}}})$
with the single component
\begin{displaymath}
\{a_{i,j},a_{i^{'},j^{'}}\}t(a_{i,j+1})(a_{i,j+1},h_{i,j+1})\cdots (a_{i,j-1},h_{i,j-1})(a_{i^{'},j^{'}+1},h_{i^{'},j^{'}+1})\cdots (a_{i^{'},j^{'}-1},h_{i^{'},j^{'}-1}).
\end{displaymath}
$X^{'}_{i,j,i,j^{'}}$ is similar to $X^{'}_{i,j,i^{'},j^{'}}$, but with heights $h_{i,j}\ and\ h_{i,j^{'}}$ interchanged;
$X^{''}_{i,j,i,j^{'}}$ is given by replacing the component $(a_{i,1},h_{i,1})\cdots (a_{i,l_{i},h_{i,l_{i}}})$ with 
\begin{align*}
\{(a_{i,j},a_{i,j^{'}})\} t(a_{i,j^{'}+1})(a_{i,j^{'}+1},h_{i,j^{'}+1})\cdots (a_{i,j-1},h_{i,j-1})\\
\&
t(a_{i,j+1})(a_{i,j+1},h_{i,j+1})\cdots (a_{i,j^{'}-1},h_{i,j^{'}-1}).
\end{align*} 

Let 
$\mathbf{N}(Q)_{\hbar}:= {\tilde{A}}/{\tilde{B}}$.
For any $X,Y \in \mathbf{N}(Q)_{\hbar}:= {\tilde{A}}/{\tilde{B}}$, 
the product of $X$ and $Y$, denoted by $X \ast Y$, is defined to be 
``placing $Y$ above $X$". More precisely,
take two elements $X,\ Y$ in the form \eqref{general form}.
Without loss of generality, we assume
the heights of the edges in $Y$ are all greater
than those in $X$. 
Then $X \ast Y$ is represented by $X\&Y$.

\begin{remark}\label{remark:general form2}
For an arbitrary element $x\in\mathbf{N}(Q)_{\hbar}$, 
$x$ can be chosen to be represented by
a linear combination of  
elements in the form \eqref{general form} 
whose heights all increase from left to right; 
that is, $h_{i,j}<h_{i',j'}$ whenever $(i, j)<(i', j')$
in the lexicographical order.
\end{remark}

\begin{proposition}[\cite{Sch2005} Section 3.6]\label{proposition:Schedler's quantization}
Suppose $Q$ is a finite quiver. Then the $\mathbb{K}[\hbar]$-algebra 
$\mathbf{N}(Q)_{\hbar}$ constructed as above quantizes the necklace Lie algbera 
$(\mathbb{K}\overline{Q})_{\natural}$.
\end{proposition}

In fact, Schedler actually showed more, namely, 
$(\mathbb{K}\overline{Q})_{\natural}$ is not only a Lie algebra, 
but also a Lie bialgebra, and $\mathbf{N}(Q)_{\hbar}$ is a Hopf algebra 
which quantizes
this Lie bialgebra.

From now on, if an algebra $A$ is a quantization of another algebra $B$, 
we will denote quantization as follows:
\begin{displaymath}
	\xymatrix{
		B \ar@{~>}[r] & A}.
\end{displaymath}

\subsection{Non-commutative reduction at quantum level}
In this subsection, we introduce the non-commutative quantum reduction. 
In the previous subsection, we already obtain the quantization and reduction of 
non-commutative cotangent bundle of $\mathbb{K}Q$ for a finite quiver $Q$. 
To complete the non-commutative ``quantization commutes with reduction", 
we have to develop the
non-commutative reduction at the quantum level; such a construction should  
fit the Kontsevich-Rosenberg principle, that is to say, the non-commutative quantum 
reduction must induce the classical quantum reduction on the quiver representations.

First, we recall a result in \cite{Sch2005}.

\begin{proposition}[\cite{Sch2005} Corollary 4.2]\label{proposition:PBW}
Let $Q$ be a finite quiver and fix the ordering $\{ x_{i}\}$ of the set of cyclic 
words in $\overline{Q}$ and idempotents in $Q_{0}$. We have:
\begin{enumerate}
\item The projection $\tilde{A} \rightarrow SL[\hbar]$ 
obtained by forgetting the heights descends to an isomorphism 
$\mathbf{N}(Q)_{\hbar} \rightarrow SL[\hbar]$ of free $\mathbb{K}[\hbar]$-modules;

\item A basis of $\mathbf{N}(Q)_{\hbar}$ as a free 
$\mathbb{K}[\hbar]$-module is given by choosing one element of the form 
(\ref{general form}) which projects to each element of the basis 
$\{x_{i_{1}}\&\cdots \&x_{i_{k}} | \text{for any k } 
\in \mathbb{Z}_{\geq 0} \text{ and } i_{1} < i_{2} < \cdots  < i_{k}\}$.
\end{enumerate}
\end{proposition}

Following the idea in \cite{Sch2005}, we fix the basis as following:
Fix the ordering $\{ x_{i}\}$ of the set of cyclic words in 
$\overline{Q}$ and idempotents in $Q_{0}$; choose 
a particular ``first element" of each $x_{i}$, 
and for any $x_{i1} \& \cdots  \& x_{i_{k}}$ with $i_{l}$ non-decreasing order,
heights can be assigned by starting with the first element to the last.
We simply denote the lifting of $x_{i_{1}}\&\cdots  \& x_{i_{k}}$ by 
$\widehat{x_{i_{1}}} \ast \cdots  \ast \widehat{x_{i_{k}}}$.
Sometimes, for a long word like $x_{1} x_{2} x_{3}\cdots $, 
we also denote the lifting of  $x_{1} x_{2} x_{3}\cdots $
by $(x_{1} x_{2} x_{3}\cdots )^{\widehat{ }}$.
Notice that $\&$ in $x_{i_{1}}\&\cdots  \& x_{i_{k}}$ is
the symmetric product in the symmetric algebra $SL[\hbar]$.

In fact, in \cite{Sch2005}, Schedler constructed an isomorphism of $\mathbb{K}$-modules,
$$l: \mathcal{U}((\mathbb{K}\overline{Q})_{\natural}) 
\rightarrow \frac{\mathbf{N}(Q)_{\hbar}}{\hbar \mathbf{N}(Q)_{\hbar}}.$$
For an arbitrary element which is expressed as combination of ordered basis 
$\sum_{i} x_{i,1}\boxtimes x_{i,2} \boxtimes\cdots  \boxtimes x_{i,r} 
\in \mathcal{U}((\mathbb{K}\overline{Q})_{\natural})$, 
$l( \sum_{i} x_{i,1}\boxtimes x_{i,2} \boxtimes\cdots  \boxtimes x_{i,r} )$ 
is represented by $\sum_{i} (\widehat{x_{i,1}} \ast \widehat{x_{i,2}} \ast \cdots \ast \widehat{x_{i,r_{i}}} ).$

\begin{definition}
Let $Q$ be a finite quiver, and let $\mathbf{N}(Q)_{\hbar}$ be the quantized necklace Lie algebra. 
The {\it non-commutative quantum moment map} is defined to be the element 
in $\mathbf{N}(Q)_{\hbar}$ represented by
\begin{displaymath}
	\widehat{\mathbf{w}}:= \sum_{a \in Q_{1}} (a,1)(a^{*},2) - (a^{*},1)(a,2).
\end{displaymath}
\end{definition}

Proposition 
\ref{proposition:noncom quantum moment map gives quantum moemnt map} 
below
justifies the above definition from the Kontsevich-Rosenberg principle point of view.

\begin{definition}\label{def:NCquantumred}
Suppose $Q$ is a finite quiver and
$\mathbf{N}(Q)_{\hbar}$ is the quantized necklace Lie algebra. 
The {\it non-commutative quantum reduction} of $\mathbf{N}(Q)_{\hbar}$ 
at $\widehat{\mathbf{w}}$ is defined to be
\begin{displaymath}
\mathcal{R}_{q}(\mathbf{N}(Q)_{\hbar},\widehat{\mathbf{w}})
 := \frac{\mathbf{N}(Q)_{\hbar}}
 {\mathbf{N}(Q)_{\hbar}\widehat{(\mathbb{K}\overline{Q}\mathbf{w})_{\natural}}\mathbf{N}(Q)_{\hbar}},
\end{displaymath}
where $\widehat{(\mathbb{K}\overline{Q}\mathbf{w})_{\natural}}$ are liftings of elements in 
$(\mathbb{K}\overline{Q}\mathbf{w})_{\natural}$, and
$\mathbf{N}(Q)_{\hbar}\widehat{(\mathbb{K}\overline{Q}\mathbf{w})_{\natural}}\mathbf{N}(Q)_{\hbar}$ 
is the two-side ideal generated by $\widehat{(\mathbb{K}\overline{Q}\mathbf{w})_{\natural}}$.
\end{definition}

Now we fix some notations.
Given a Lie algebra $\mathfrak{g}$, denote its universal 
enveloping algebra by $\mathcal{U} \mathfrak{g}$. 
Define a map 
\begin{displaymath}
v:\, \mathcal{U}((\Pi Q)_{\natural})\, \rightarrow\, 
\frac{\mathcal{R}_{q}(\mathbf{N}(Q)_{\hbar},\widehat{\mathbf{w}})}
{\hbar \mathcal{R}_{q}(\mathbf{N}(Q)_{\hbar},\widehat{\mathbf{w}})}
\end{displaymath}
by\begin{displaymath}
\ \sum_{i}c_{i}pr(x_{i1})\boxtimes \cdots  
\boxtimes pr(x_{i k_{i}})\, \mapsto\, \sum_{i}c_{i} 
[ \overline{\widehat{x_{i1}} \ast \cdots \ast \widehat{x_{i k_{i}}} } ],
\end{displaymath}
where 
$pr:(\mathbb{K}\overline{Q})_{\natural} \rightarrow (\Pi Q)_{\natural}$ is the canonical homomorphism induced from the projection of algebras 
$\mathbb{K}\overline{Q} \rightarrow \Pi Q$,
$\overline{\widehat{x_{i1}} \ast \cdots \ast \widehat{x_{i k_{i}}} }$ 
is the class in 
$\mathcal{R}_{q}(\mathbf{N}(Q)_{\hbar},\widehat{\mathbf{w}})$ 
represented by $\widehat{x_{i1}} \ast \cdots \ast \widehat{x_{i k_{i}}}$,
$[ \overline{\widehat{x_{i1}} \ast \cdots \ast \widehat{x_{i k_{i}}} } ]$
is the class in $\displaystyle\frac{\mathcal{R}_{q}(\mathbf{N}
(Q)_{\hbar},\widehat{\mathbf{w}})}{\hbar \mathcal{R}_{q}
(\mathbf{N}(Q)_{\hbar},\widehat{\mathbf{w}})}$ 
represented by $\overline{\widehat{x_{i1}}\ast \cdots \ast \widehat{x_{i k_{i}}}}$. 
Once again, readers are reminded that 
$\widehat{x_{i1}} \ast \cdots \ast \widehat{x_{i k_{i}}}$ is the product of 
$\widehat{x_{i1}},\cdots ,\widehat{x_{i k_{i}}}$ in $\mathbb{N}(Q)_{\hbar}$, and $\boxtimes$ 
is the  multiplication in $\mathcal{U}((\Pi Q)_{\natural})$.

\begin{lemma}\label{lemma: w is stable}
Let $Q$ be a finite quiver and 
$\mathbf{w} = \sum_{a \in Q_{1}} [a,a^{*}]$. 
For any cyclic path $x \in (\mathbb{K} \overline{Q})_{\natural}$ 
and any element 
$\sum_{i}y_{i} \mathbf{w} \in (\mathbb{K} \overline{Q} \mathbf{w})_{\natural}$, 
we have
\begin{displaymath}
		\{x, \sum_{i} y_{i}\mathbf{w}\} \in (\mathbb{K} \overline{Q} \mathbf{w})_{\natural}.
\end{displaymath}
\end{lemma}

\begin{proof}
See \cite[Proposition 4.4.3]{CBEG2007}.
\end{proof}

\begin{theorem}\label{theorem:R(A) quantizes preproj algebra}
Suppose $Q$ is a finite quiver and $\mathcal{R}_{q}(\mathbf{N}(Q)_{\hbar},
\widehat{\mathbf{w}})$ 
is the non-commutative quantum reduction of $\mathbf{N}(Q)_{\hbar}$. Then
\begin{displaymath}
v:\mathcal{U}( (\Pi Q)_{\natural} )\cong
\frac{\mathcal{R}_{q}(\mathbf{N}(Q)_{\hbar},
\widehat{\mathbf{w}})}{\hbar \mathcal{R}_{q}(\mathbf{N}(Q)_{\hbar},
\widehat{\mathbf{w}})}
\end{displaymath}
as $\mathbb{K}$-modules.
\end{theorem}

\begin{proof}
First, we check $v$ is well-defined. We choose an arbitrary 
element $\sum_{i}c_{i}pr(x_{i1})\boxtimes \cdots  \boxtimes pr(x_{i k_{i}})$,
and since $pr$ is not an isomorphism, we need to make sure 
the value of $v$ is independent of $\ker pr$. 
For each $i,j$, $y_{ij} - x_{ij} \in  \ker pr = 
(\mathbb{K}\overline{Q}\mathbf{w})_{\natural}$, by definition of 
$\mathcal{R}_{q}(\mathbf{N}(Q)_{\hbar},\mathbf{w})$, 
we have $\overline{\widehat{x_{i1}}\cdots \widehat{x_{i k_{i}}} } 
=\overline{\widehat{y_{i1}} \ast \cdots \ast \widehat{y_{i k_{i}}} }$ for each $i$.
Thus $v(\sum_{i}c_{i}pr(x_{i1})\boxtimes\cdots\boxtimes pr(x_{i k_{i}})) 
= v \Big(\sum_{i}c_{i}pr(y_{i1})\boxtimes \cdots  \boxtimes pr(y_{i k_{i}}) \Big)$.

Second, we check $v$ is surjective. 
On one hand, we have a surjective morphism 
$$
\frac{\mathbf{N}(Q)_{\hbar}}
{\hbar \mathbf{N}(Q)_{\hbar}} \rightarrow \frac{\mathcal{R}_{q}
(\mathbf{N}(Q)_{\hbar},\widehat{\mathbf{w}})}{\hbar \mathcal{R}_{q}
(\mathbf{N}(Q)_{\hbar},\widehat{\mathbf{w}})}
$$ 
which is induced from 
canonical surjective morphism $\mathbf{N}(Q)_{\hbar} 
\rightarrow \mathcal{R}_{q}(\mathbf{N}(Q)_{\hbar},\widehat{\mathbf{w}})$. 
On the other hand, by Proposition \ref{proposition:Schedler's quantization} 
we know $\mathcal{U}((\mathbb{K}\overline{Q})_{\natural}) \cong 
\displaystyle\frac{\mathbf{N}(Q)_{\hbar}}{\hbar \mathbf{N}(Q)_{\hbar}}$.
Combine these two morphisms, we have a commutative diagram
\begin{displaymath}
\xymatrix{
\displaystyle\frac{\mathbf{N}(Q)_{\hbar}}{\hbar \mathbf{N}(Q)_{\hbar}} \ar@{->>}[d] 
& \mathcal{U}((\mathbb{K}\overline{Q})_{\natural}) \ar[l]_-{\cong} \ar@{->>}^{pr}[d] \\
\displaystyle\frac{\mathcal{R}_{q}
	(\mathbf{N}(Q)_{\hbar},\widehat{\mathbf{w}})}
{\hbar\mathcal{R}_{q}
	(\mathbf{N}(Q)_{\hbar},\widehat{\mathbf{w}})} & \mathcal{U}(\Pi Q) \ar[l]_{v},
}
\end{displaymath}
which implies that $v$ is surjective.

Finally, we check $v$ is injective. Suppose we have an arbitrary element 
$\sum_{i}c_{i}pr(x_{i1})\boxtimes\cdots\boxtimes pr(x_{i k_{i}})$
such that $v \Big(\sum_{i}c_{i}pr(x_{i1})\boxtimes \cdots  \boxtimes pr(x_{i k_{i}}) \Big) 
= \sum_{i}c_{i} [ \overline{\widehat{x_{i1}} \ast \cdots \ast \widehat{x_{i k_{i}}} } ] = 0$,
we need to show this element is zero.

There are only two cases to check. The first case is that
$\sum_{i}c_{i}  \overline{\widehat{x_{i1}} \ast \cdots \ast \widehat{x_{i k_{i}}} } = 0$ in 
$\mathcal{R}_{q}(\mathbf{N}(Q)_{\hbar},\widehat{\mathbf{w}})$; 
that is to say, $\sum_{i}c_{i} \widehat{x_{i1}} \ast \cdots \ast 
\widehat{x_{i k_{i}}} \in {\mathbf{N}(Q)_{\hbar}\widehat{(\mathbb{K}\overline{Q}\mathbf{w})_{\natural}}\mathbf{N}(Q)_{\hbar}}$.
Then we can rewrite this element as 
$$\sum_{i}c_{i} \widehat{x_{i1}} \ast \cdots \ast \widehat{x_{i k_{i}}} 
= \sum_{i} \widehat{\alpha_{i,1}} \ast \cdots \ast \widehat{\alpha_{i,r_{i}}} 
\ast {\Big(\sum_{b_{1}\cdots b_{s}}b	_{1}\cdots b_{s} \sum_{a \in Q_{1}}
 [a,a^{*}] \Big)^{\widehat{ }}}\ast 
 \ \widehat{\beta_{i,1}} \ast \cdots \ast \widehat{\beta_{i,t_{i}}},$$
here all $\alpha_{i,j},\ \beta_{i,j}$ are elements in $(\mathbb{K}\overline{Q})_{\natural}$, 
and $b_{1}\cdots b_{s}$ is a cyclic path in $\overline{Q}$. 
Here, we use $[\sum_{i}c_{i} \widehat{x_{i1}} \ast \cdots \ast \widehat{x_{i k_{i}}}]$ to 
represent the equivalence class in 
$\displaystyle \frac{\mathbf{N}(Q)_{\hbar}}{\hbar \mathbf{N}(Q)_{\hbar}}$.
Since the lifting 
$$l:\mathcal{U}((\mathbb{K}\overline{Q})_{\natural}) 
\rightarrow \frac{\mathbf{N}(Q)_{\hbar}}{\hbar \mathbf{N}(Q)_{\hbar}}$$ 
is an isomorphism, we have that
\begin{align*}
&l^{-1} \Big(\big[\sum_{i}c_{i} \widehat{x_{i1}} \ast \cdots \ast \widehat{x_{i k_{i}}}\big] \Big) \\
& = l^{-1}\Big(
\big[\sum_{i} \widehat{\alpha_{i,1}} \ast \cdots \ast \widehat{\alpha_{i,r_{i}}} 
{\ast \big( \sum_{b_{1}	\cdots b_{s}}b_{1}\cdots b_{s} \sum_{a \in Q_{1}} [a,a^{*}] 
\big)^{\widehat{ }}}\ast \ \widehat{\beta_{i,1}} \ast \cdots \ast \widehat{\beta_{i,t_{i}}}\big] \Big) 
\end{align*}
implies
\begin{align*}	
&\sum_{i}c_{i} {x_{i1}}\boxtimes\cdots \boxtimes{x_{i k_{i}}} \\
& = \sum_{i} l^{-1}[\widehat{\alpha_{i,1}} \ast \cdots \ast \widehat{\alpha_{i,r_{i}}} 
\ast{\big(\sum_{b_{1}\cdots b_{s}} b_{1}\cdots b_{s} \sum_{a \in Q_{1}} [a,a^{*}]\big)^{\widehat{ }}} 
\ast \  \widehat{\beta_{i, 1}} \ast \cdots \ast \widehat{\beta_{i,t_{i}}}].
\end{align*}
If we apply $pr$ to the left hand side, we get
$\sum_{i}c_{i} pr(x_{i 1}) \boxtimes \cdots  \boxtimes pr(x_{i k_{i}})$. 
To show this is zero we need to show the right hand side is zero.

In fact, by Proposition \ref{proposition:PBW}, 
we can rewrite $$\widehat{\alpha_{i,1}} \ast \cdots \ast \widehat{\alpha_{i,r_{i}}} 
\ast{\big(\sum_{b_{1}\cdots b_{s}} b_{1}\cdots b_{s} \sum_{a \in Q_{1}} [a,a^{*}]\big)^{\widehat{ }}} 
\ast \  \widehat{\beta_{i, 1}} \ast \cdots \ast \widehat{\beta_{i,t_{i}}}$$
to be the linear combination of the basis, 
and due to Lemma \ref{lemma: w is stable},
we see that each summand must contain an element in 
$\widehat{(\mathbb{K}\overline{Q} \mathbf{w})_{\natural}}$.
Thus by definition of $pr$, the right hand side is zero, 
and hence we do have 
$\sum_{i}c_{i} pr(x_{i 1}) \boxtimes \cdots  \boxtimes pr(x_{i k_{i}}) = 0$.

For the second case, that is $$\sum_{i}c_{i} 
\overline{\widehat{x_{i1}} \ast \cdots \ast \widehat{x_{i k_{i}}} } 
\neq 0\quad\mbox{but}\quad
[\sum_{i}c_{i}  \overline{\widehat{x_{i1}} \ast \cdots \ast \widehat{x_{i k_{i}}} }] = 0.$$ 
This means 
$\sum_{i}c_{i}  \overline{\widehat{x_{i1}} \ast \cdots \ast \widehat{x_{i k_{i}}} } 
\in \hbar \mathcal{R}_{q}(\mathbb{N}(Q),\widehat{\mathbf{w}})$, 
which 
contradicts to the fact 
$\sum_{i}c_{i}  \overline{\widehat{x_{i1}} \ast \cdots \ast \widehat{x_{i k_{i}}} }$ 
does not contain $\hbar$.

Combining the above three facts, we get
an isomorphism \[
\mathcal{U}( (\Pi Q)_{\natural} )\cong
\frac{\mathcal{R}_{q}(\mathbf{N}(Q)_{\hbar},
\widehat{\mathbf{w}})}{\hbar \mathcal{R}_{q}
(\mathbf{N}(Q)_{\hbar},\widehat{\mathbf{w}})}  .\qedhere\]
\end{proof}

In summary, 
we construct in Definition \ref{def:NCquantumred} 
the non-commutative quantum reduction 
$\mathcal{R}_{q}(\mathbf{N}(Q)_{\hbar},\widehat{\mathbf{w}})$
of $\mathbf{N}(Q)_{\hbar}$,
and prove in Theorem \ref{theorem:R(A) quantizes preproj algebra}
that this reduction 
is the quantization of the corresponding preprojective algebra. 
In other words, we now have the non-commutative version of 
``quantization commutes with reduction" for quiver algebras, 
which is summarized as follows.

\begin{theorem}[Noncommutative quantization commutes with reduction]\label{thm:NCQR}
Suppose $Q$ is a finite quiver. Then we have the following commutative diagram:
\begin{equation}\label{eq:quantizationcommwithreduction}\begin{split}
\xymatrixcolsep{4pc}
\xymatrix{
\mathbf{N}(Q)_{\hbar} \ar@{-->}[r]  
& \mathcal{R}_{q}(\mathbf{N}(Q)_{\hbar}, \widehat{\mathbf{w}}) \\
(\mathbb{K}\overline{Q})_{\natural}	\ar@{-->}[r]
\ar@{~>}[u]
& (\Pi Q) _{\natural} .\ar@{~>}[u]
}
\end{split}
\end{equation}
\end{theorem}

\section{Quantization of quiver representations}\label{sect:QR}

In this section, we collect the results of Holland \cite{Hol1999}
on the commutativity of quantization and reduction
for quiver representations. Since we will heavily use 
his results, in what follows we give some details.

\subsection{Linear differential operators}

Let us start by recalling the following.

\begin{definition} \label{definition:differential operators}
Let $X$ be a smooth affine variety with the coordinate ring 
$T:=\mathbb{K}[X]$. The {\it algebra of linear differential operators} on 
$X$ is the subalgebra of $\mathrm{End}_{\mathbb{K}}T$ 
generated by $\mathrm{End}_{T}T\cong T$ 
and $\mathrm{Der}_{\mathbb{K}}(T)$.
\end{definition}

Let $T$ be as above. Suppose $V$ is a finite dimensional $\mathbb{K}$-vector space. Let
$E_{V}:= V \otimes T$; then $\mathrm{End}_{\mathbb{K}}V \otimes \mathcal{D}(T)$ 
embeds naturally in $\mathrm{End}_{\mathbb{K}}E_{V}$, 
which is denoted by $\mathcal{D}(E_{V})$. 
By definition, $\mathcal{D}(E_{\mathbb{K}}) = \mathcal{D}(T)$.

The following two propositions will be used later, whose proof
can be found in Holland \cite[Section 2]{Hol1999}.

\begin{proposition}\label{prop:filtration of DO}
Suppose $V$ is a finite dimensional $\mathbb{K}$-vector space.
Then $\mathcal{D}(E_{V})$ is equipped with a natural filtration:
\begin{displaymath}
\mathcal{D}^{0}(E_{V}) \subset  \mathcal{D}^{1}(E_{V}) 
\subset \cdots  \subset \mathcal{D}^{n}(E_{V}) \subset \cdots 
\end{displaymath}
where
\begin{displaymath}
\mathcal{D}^{0}(E_{V}) = \mathrm{End}_{\mathbb{K}}V \otimes T,\ 
\mathcal{D}^{1}(E_{V})= \mathrm{End}_{\mathbb{K}}V \otimes (T+\mathrm{Der}_{\mathbb{K}}T)	\end{displaymath}
and $\mathcal{D}^{n}(E_{V}) = \mathcal{D}^{n-1}(E_{V}) \mathcal{D}^{1}(E_{V})$
for $n\ge 2$.
\end{proposition}

\begin{proposition}\label{proposition:DO quantizes cotangent bundle}
Suppose $X$ is an smooth affine variety and
$V$ is a finite dimensional vector space. Then
\begin{enumerate}
\item for $n>1$, $\displaystyle\frac{\mathcal{D}^{n}(E_{V})}{\mathcal{D}^{n-1}(E_{V})} 
	\cong \mathrm{End}_{\mathbb{K}}V \otimes Sym^{n}(\mathrm{Der}_{\mathbb{K}}T)$;
\item $gr\mathcal{D}(E_{V}) \cong \mathrm{End}_{\mathbb{K}}V 
\otimes \mathbb{K}[T^{*}X]$. In particular, $gr(\mathcal{D}(T)) \cong \mathbb{K}[T^{*}X]$.
		Here $gr(-)$ means the associated graded algebra. 
	\end{enumerate}
\end{proposition}

Let $G$ be a reductive algebraic group with Lie algebra $\mathfrak{g}$. 
Suppose $G$ acts on $X$, then we get a Lie algebra map 
$\tau:\, \mathfrak{g}\, \rightarrow\, \mathrm{Der}_{\mathbb{K}}T$. 
Furthermore, one can extend $\tau$ to be an algebra map
$\tau:\, \mathcal{U}( \mathfrak{g} ) \, \rightarrow \, \mathcal{D}(T)$.
Define a $G$-action on $\mathcal{D}(T)$ by
\begin{equation} \label{formula: group action on differential operators}
	g.D:=g\circ D \circ g^{-1},\;\; \text{for any }g \in \mathrm{G},\ D \in \mathcal{D}(T).
\end{equation}
This $\mathrm{G}$-action induces a $\mathfrak{g}$-action on $\mathcal{D}(T)$ by:
\begin{equation}\label{formula: Lie algebra action on differential operators}
x.D:= \tau(x) D-D\tau(x) = [\tau(x),D],\;\;
\text{for any }x\in \mathfrak{g},\ D\in \mathcal{D}(T).
\end{equation}

Now consider the following function:
\begin{displaymath}
	p:\, \mathbb{N}^{Q_{0}}\, \rightarrow\, \mathbb{Z}, \alpha \mapsto 1 + \sum_{a \in Q_{1}} \alpha_{s(a)}\alpha_{t(a)} - \sum_{i \in Q_{0}} \alpha_{i} ^{2}.
\end{displaymath}
The following proposition, due to Crawley-Boevey, 
gives several equivalent descriptions of a {\it flat} moment map:

\begin{proposition}[Crawley-Boevey \cite{Cra2001}, Theorem 1.1]
Suppose $Q$ is a finite quiver, $\mathbf{d} \in \mathbb{N}^{Q_{0}}$ 
is a dimension vector, and 
$\mu:\, T^{*}\mathrm{Rep}(Q,\mathbf{d})\, 
\rightarrow\, \mathfrak{gl}_{\mathbf{d}}$ is the moment map. 
Then the following statements are equivalent:
\begin{enumerate}
\item $\mu$ is a flat morphism;
		
\item $\mu^{-1}(0)$ is of dimension $\sum_{i \in Q_{0}} d_{i}^{2} -1 + 2p(\mathbf{d})$;
		
\item $\dim \mu^{-1}(0) = 2\dim\mathrm{Rep}(Q,\mathbf{d}) 
- \dim \mathrm{GL}_{\mathbf{d}}$.
\end{enumerate}
\end{proposition}

Now take a character $\chi: \mathfrak{g}\to\mathbb K$
which does not vanish on $\ker\tau$. The following two Propositions
are obtained by Holland in \cite{Hol1999}, which will be used later:

\begin{proposition}[\cite{Hol1999} Proposition 2.4]
Suppose that the moment map $\mu:\, T^{*}X\, \rightarrow\, \mathfrak{g}^{*}$ is flat. Then
\begin{displaymath}
gr\Big(\frac{\mathcal{D}(E_{V})}{\mathcal{D}(E_{V})(\tau - \chi)
(\mathfrak{g})} \Big) \cong \frac{gr(\mathcal{D}(E_{V}))}
{\big(gr \mathcal{D}(E_{V})\big)\mathfrak{g}}
\end{displaymath} 
as $\mathbb{K}[T^{*}X]$-modules.
In particular,
\begin{displaymath}
gr( \frac{(\mathcal{D}(E_{V}))^{{G}}}{\big(\mathcal{D}(E_{V})
(\tau-\chi)(\mathfrak{g}) \big)^{{G}}} )  \cong 
\big(\mathrm{End}_{\mathbb{K}}V \otimes \mathbb{K}[\mu^{-1}(0)]\big)^{{G}}
\end{displaymath}
as $\mathbb{K}$-algebras,
where $\tau$ is described after Proposition \ref{proposition:DO quantizes cotangent bundle}
and $\chi: \mathfrak{g} \rightarrow \mathbb{K}$ is a character of $\mathfrak{g}$,
which does not vanish on $\ker\,\tau$.
\end{proposition}

\begin{proposition}[\cite{Hol1999} Proposition 2.4]
\label{proposition:DO quantizes quiver variety}
With notations as above, there exists an isomorphism
\begin{displaymath}
gr \Big(\frac{(\mathcal{D}(T))^{{G}}}
{ \big(\mathcal{D}(T)(\tau-\chi)
(\mathfrak{g}) \big)^{{G}}} \Big) \cong 
\big(\mathbb{K}[\mu^{-1}(0)]\big)^{{G}}.
\end{displaymath}
\end{proposition}

\subsection{Quantization of quiver representations}

Now we are ready to study the quantization and reduction of quiver representations.

\begin{definition}\label{def:filteredquant}
Suppose $A$ is commutative $\mathbb{Z}_{\geq 0}$-graded 
$\mathbb{K}$-algebra, equipped with a Poisson bracket 
whose degree is $-1$.
A {\it filtered quantization} of $A$ is a filtered algebra 
$\mathcal{A}_{\leq \bullet}$, such that $gr \mathcal{A}_{\leq \bullet}$ 
is isomorphic to $A$ as graded Poisson algebras.
\end{definition}

Notice that the Poisson bracket on $gr \mathcal{A}_{\leq \bullet}$ is given as follows:
for any $\overline{a},\ \overline{b} \in gr\mathcal{A}_{\leq \bullet}$,
\begin{displaymath}
\{\overline{a},\overline{b}\}:= \overline{[a,b]}
\end{displaymath}
Here $[a,b]$ is the commutator with respect to the product in $\mathcal{A}_{\leq \bullet}$.

By the above definition and Proposition \ref{proposition:DO quantizes quiver variety}, we obtain
the following:

\begin{corollary}\label{coro:filtered quant}
Suppose $Q$ is a finite quiver, $\mathbf{d}$ is a dimension vector 
and $\chi$ is a character of $\mathfrak{gl}_{\mathbf{d}}$ 
such that the moment map 
$\mu:\, T^{*}\mathrm{Rep}(Q,\mathbf{d})\, \rightarrow \, 
\mathfrak{gl}_{\mathbf{d}}$ is a flat morphism and $\chi$ 
does not vanish on $\mathrm{Ker} \tau$.
Then 
\begin{enumerate}
\item[$(1)$]$\mathcal{D}(\mathrm{Rep}
(Q,\mathbf{d}))$ quantizes $\mathbb{K}[T^{*}\mathrm{Rep}(Q,\mathbf{d})]$, 
and

\item[$(2)$]
$\displaystyle\frac{(\mathcal{D}(\mathrm{Rep}
(Q,\mathbf{d})))^{\mathrm{GL}_{\mathbf{d}}}}
{(\mathcal{D}(\mathrm{Rep}(Q,\mathbf{d}))(\tau-\chi)
(\mathfrak{gl}_{\mathbf{d}}))^{\mathrm{GL}_{\mathbf{d}}}}$ 
quantizes ${\mu^{-1}(0)}//\mathrm{GL}_{\mathbf{d}}$.
\end{enumerate}
\end{corollary}

For the convenience of our argument below,
we also need the {\it graded quantization}, which we now recall.

\begin{definition}\label{def:gradedquant}
Suppose $A$ is a $\mathbb{Z}_{\ge 0}$-graded 
commutative $\mathbb{K}$-algebra equipped 
with Poisson bracket $\{,\}$ of degree $-1$. 
A {\it graded quantization} of $A$ is a graded $\mathbb{K}[\hbar]$-algebra($\mathrm{deg}\hbar = 1$) 
which is free as a $\mathbb{K}[\hbar]$-module, 
equipped with an isomorphism of $\mathbb{K}$-algebras:
	\begin{displaymath}
		\Phi: \frac{A_{\hbar}}{\hbar A_{\hbar}}\, \rightarrow\, A,
	\end{displaymath}
such that for any $a,b \in A_{\hbar}$, if we denote their images
in $\displaystyle\frac{A_{\hbar}}{\hbar A_{\hbar}}$
by $\overline{a},\overline{b}$, then 
$\Phi \big(\overline{\displaystyle\frac{1}{\hbar}(ab - ba)} \big) 
= \{\Phi(\overline{a}),\Phi(\overline{b})\}$.
\end{definition}

Now let us recall the notion of the homogeneous differential operators
(see Losev \cite{Los2012}):

\begin{definition}\label{def:homogenized DO}
Let $X$ be a smooth affine variety, 
and $T:= \mathbb{K}[X]$ be its coordinate ring. 
The homogeneous differential operators on $X$ is 
defined to be the graded $\mathbb{K}[\hbar]$-algebra 
$\mathcal{D}_{\hbar}(X)$ generated by $T$ equipped with 
degree $0$, $\mathrm{Der}_{\mathbb{K}}(T)$ equipped with 
degree $1$ and subject to the following relation:
\begin{enumerate}
\item for any $f,g \in T$, $f \ast g = fg;$
		
\item for $X \in \mathrm{Der}_{\mathbb{K}}T,
		\ f \in T$, $f \ast X = fX$;
		
\item $X \in \mathrm{Der}_{\mathbb{K}}T,
		\ f \in T$;
		$X \ast f = fX + \hbar X(f)$;
		
\item for $X,\ Y \in \mathrm{Der}_{\mathbb{K}}T$, 
		$X \ast Y - Y \ast X = \hbar [X,Y]$.
\end{enumerate} 
\end{definition}

With the above definition, Corollary \ref{coro:filtered quant} is rephrased as follows:

\begin{corollary}
Suppose $Q$ is a finite quiver, $\mathbf{d}$ is a dimension vector and 
$\chi$ is a character of $\mathfrak{gl}_{\mathbf{d}}$ such that 
the moment map $\mu:\, T^{*}\mathrm{Rep}(Q,\mathbf{d})\, 
\rightarrow \, \mathfrak{gl}_{\mathbf{d}}$ is a flat morphism and 
$\chi$ does not vanish on $\ker \tau$.
Then
\begin{enumerate}
\item[$(1)$]
 $\mathcal{D}_{\hbar} (\mathrm{Rep}(Q,\mathbf{d}))$ 
is a graded quantization of $\mathbb{K}[T^{*} \mathrm{Rep}(Q,\mathbf{d})]$,
and 
	
\item[$(2)$] $\displaystyle\frac{ 
\big(\mathcal{D}_{\hbar}(\mathrm{Rep}(Q,\mathbf{d})) \big)^{\mathrm{GL}_{\mathbf{d}}}}
{ \big(\mathcal{D}_{\hbar}(\mathrm{Rep}(Q,\mathbf{d}))(\tau-\hbar\chi)
(\mathfrak{gl}_{\mathbf{d}}) \big)^{\mathrm{GL}_{\mathbf{d}}}}$
is a graded quantization of $\mathbb{K}[\mathcal{M}_{\mathbf{d}}(Q)]$.
\end{enumerate}
\end{corollary}

\begin{remark}
The above quantization is consistent with Losev \cite{Los2012}.
\end{remark}

\begin{definition}[Quantum moment map]\label{definition:quantum moment map}
Let $G$ be an algebraic group with Lie algebra $\mathfrak{g}$; 
let $A_{\hbar}$ be a graded $\mathbb{K}[\hbar]$-algebra 
equipped with a $\mathfrak{g}$-action. The map 
$\mu^{\sharp}_{\hbar}:\, \mathcal{U}_{\hbar}\mathfrak{g}\, 
\rightarrow\, A_{\hbar}$ is called a {\it quantum moment map}
if $\mu_{\hbar}^{\sharp}(\mathfrak{g}) \subset (A_{\hbar})_{1}$
and for any $v \in \mathfrak{g}$, 
$$A_{\hbar}\, \rightarrow\, A_{\hbar},\; a\mapsto\frac{1}{\hbar}[\mu_{\hbar}^{\sharp}(v),a]\, 
$$ 
is the $\mathfrak{g}$-action of 
$v$.
\end{definition}

Here 
$$\mathcal{U}_{\hbar} \mathfrak{g} := T_{\mathbb{K}[\hbar]}
(\mathbb{K}[\hbar] \otimes \mathfrak{g})/(\{x \otimes y - y \otimes x 
- \hbar[x,y] | \text{ for any } x, y \in \mathfrak{g}\}).$$
The above definition coincides with the quantum moment map 
introduced by Lu in \cite{Lu1993} (see also Xu \cite{Xu1998} and Losev \cite{Los2012}). 
Let us briefly recall the main idea in loc cit.
Suppose $(X,G)$ is a Hamiltonian $G$-space. 
Let $H$ be a Hopf algebra quantizing $G$ and let $V$ be a quantization 
of $X$ and equipped with $H$-action. 
Lu proposed that $\Phi:H \rightarrow V$ is called the {\it quantum moment map} if  
$$a.v=\Phi(S(a^{(1)}))v\Phi(a^{(2)}),$$
where $a\in H,\ v \in V$, $S$ is the antipode of $H$ and 
coproduct of $a$ is written as $a^{(1)} \otimes a^{(2)}$.

Now, fix a dimension vector $\mathbf{d}$;
suppose $Q$ is a finite quiver with a flat moment map 
$\mu: T^{*}\mathrm{Rep}(Q,\mathbf{d})\, \rightarrow\, \mathfrak{gl}_{\mathbf{d}}$. 
Differentiating the $\mathrm{GL}_{\mathbf{d}}$-action on $\mathrm{Rep}(Q,\mathbf{d})$, 
one has (see also Holland \cite[Lemma 3.1]{Hol1999}):

\begin{lemma}\label{lemma:gl_d action on Rep}
The $\mathfrak{gl}_{\mathbf{d}}$-action on $\mathrm{Rep}(Q,\mathbf{d})$ 
is given by the  Lie algebra homomorphism 
$\tau:\, \mathfrak{gl}_{\mathbf{d}}\, \rightarrow\, 
\mathcal{D}(\mathrm{Rep}(Q,\mathbf{d}))$, which maps
\begin{displaymath}
e^{i}_{pq}\, \mapsto\, \sum_{a\in Q,s(a)=i} \sum_{j=1}^{d_{t(a)}} [a]_{jp} 
\frac{\partial}{\partial (a)_{jq}} - \sum_{a\in Q,t(a)=i}
\sum_{j=1}^{d_{s(a)}} [a]_{qj}\frac{\partial }{\partial (a)_{pj}},
\end{displaymath}
where
$e^{i}_{pq}$ is the $(q,p)$-th elementary matrix in the $i$-the summand of
$\mathfrak{\mathfrak{gl}_{\mathbf{d}}}$.
\end{lemma}

By this lemma, the quantum moment map in the case of quiver varieties 
is given by $\mu^{\sharp}_{\hbar} = \tau - \hbar \chi $ for some characters $\chi$
not vanishing on $\ker\tau$.
Thus
$$\frac{ \big(\mathcal{D}_{\hbar}(\mathrm{Rep}(Q,\mathbf{d})) \big)^{\mathrm{GL}_{\mathbf{d}}}}
{ \big(\mathcal{D}_{\hbar}(\mathrm{Rep}(Q,\mathbf{d}))(\tau-\hbar\chi)
(\mathfrak{gl}_{\mathbf{d}}) \big)^{\mathrm{GL}_{\mathbf{d}}}}
$$ 
obtained from $\mathcal{D}_{\hbar}(\mathrm{Rep}(Q,\mathbf{d}))$ is a
quantum Hamiltonian reduction in the sense of Definition \ref{definition:quantum moment map}.

Summarizing the above argument, we have 
the following theorem, due to Holland \cite{Hol1999}
(see also Losev \cite[Lemma 3.3.1]{Los2012} for more details), saying that 
on quiver representations
``the quantization commutes with the reduction".
  
\begin{theorem}[Quantization commutes with reduction] \label{theoremofholland}
Let $Q$ be a finite quiver and $\mathbf{d}$ a dimension vector 
such that the moment map $\mu$ is flat. Let $\tau$ be 
the Lie algebra homomorphism in Lemma \ref{lemma:gl_d action on Rep}. 
For a character $\chi:\ \mathfrak{gl}_{\mathbf{d}} \rightarrow \mathbb{K}^{*}$ 
not vanishing on $\ker\tau$, we have a commutative diagram 
\begin{equation}\label{eq:NCquantizationcommwithreduction}\begin{split}
\xymatrixcolsep{1.5cm}\xymatrix{
\mathcal{D}_{\hbar}(\mathrm{Rep}(Q,\mathbf{d})) \ar@{-->}[r]  
& \displaystyle
\frac{ \big(\mathcal{D}_{\hbar}(\mathrm{Rep}(Q,\mathbf{d})) \big)^{\mathrm{GL}_{\mathbf{d}}}}
{ \big(\mathcal{D}_{\hbar}(\mathrm{Rep}(Q,\mathbf{d}))(\tau - \hbar \chi)
(\mathfrak{gl}_{\mathbf{d}}) \big)^{\mathrm{GL}_{\mathbf{d}}}} \\
\mathbb{K}[T^{*}\mathrm{Rep}(Q,\mathbf{d})] \ar@{-->}[r] \ar@{~>}[u]
& \mathbb{K}[\mathcal{M}_{\mathbf{d}}(Q)]. \ar@{~>}[u] }
\end{split}
\end{equation}
\end{theorem}

\section{Trace maps and proof of the main theorem}\label{sect:fromNC}

In this section, we study two trace maps, at the classical and 
the quantum level respectively, 
connecting
the previously obtained two commutative diagrams
\eqref{eq:quantizationcommwithreduction}
and
\eqref{eq:NCquantizationcommwithreduction},
from which we obtain that ``the quantization commutes with the reduction" 
of quiver algebras
fits the Kontsevich-Rosenberg principle.

\subsection{Quantum trace maps}

The classical trace map on the representation spaces
is introduced in \S\ref{Sect:Repspaces}. In this subsection
we recall the {\it quantum trace map}, due to Schedler, 
from the quantized necklace Lie algebras to the quantization
of quiver representation spaces.

\begin{definition}[Schedler \cite{Sch2005} Section 3.4] \label{definition:quantum trace map}
Suppose $Q$ is a finite quiver, $\mathbf{d}$ is a dimension vector,  
and
$\mathbf{N}(Q)_{\hbar}$ is the quantized necklace Lie algebra. 
Then {\it quantum trace map},
denoted by $\mathrm{Tr}^q$, is a $\mathbb{K}[\hbar]$-linear map 
from $\mathbf{N}(Q)_{\hbar}$ to $\mathcal{D}_{\hbar}
(\mathrm{Rep}(Q,\mathbf{d}))$ such that 
for any element in the form (\ref{general form}), its image is
\begin{equation}
d_{v1}\cdots d_{v_{m}} \sum^{d_{s(a_{i,j})}} _{\forall i,j\,k_{i,j}=1}
\left(\prod_{h=1} ^{N} [a_{\phi ^{-1}(h)}]_{k_{\phi^{-1}(h)},
k_{\phi^{-1}(h)+1}}\right),\label{F.1}
\end{equation}
where $\{ h_{i,j} \}=\{1,2,\cdots ,N\}$, 
$\phi$ is the map such that $\phi(i,j)=h_{i,j}$ 
and for $x\in Q$, $[x]_{i,j}$ is the $(i,j)$-th coordinate function
and $[x^{*}]_{i,j}$ is the
differential operator $\displaystyle\frac{\partial}{\partial [x]_{j,i}}$. 
\end{definition}

It is proved in Schedler \cite[Section 3.4]{Sch2005} that $\mathrm{Tr}^q$
in the above definition does not depend on choice
of the representatives of the elements in $\mathbf{N}(Q)_{\hbar}$.

From the above definition, 
if the heights in \eqref{general form} increase 
from the left to the right, 
then its image under $\mathrm{Tr}^{q}$ is
\begin{equation}\label{quantum trace formula}
d_{v1}\cdots d_{v_{m}} \prod_{i=1} tr\left(\prod_{j=1}[a_{ij}]\right).
\end{equation}
With this formula, we also see that the images of quantum trace map are 
$\mathrm{GL}_{\mathbf{d}}$-invariant; in other words, we have the following:

\begin{proposition}\label{prop:image q.tr G-inv}
Let $Q$ be a finite quiver, and $\mathbf{d}$ is a dimension vector. 
Then we have $$\mathrm{Tr}^{q} (\mathbf{N}(Q)_{\hbar}) 
\subset (\mathcal{D}_{\hbar}(\mathrm{Rep}(Q,\mathbf{d})))^{\mathrm{GL}_{\mathbf{d}}}.$$
\end{proposition}

\begin{proof}
Without loss of generality, we choose 
	$$
	X = \widehat{x_{1}} \ast \widehat{x_{2}} \ast \cdots \ast \widehat{x_{r}}
	$$
where for each $i$, $x_{i} = x_{i,1}\cdots x_{i,j_{i}} \in (\mathbb{K}\overline{Q})_{\natural}$. 
By \eqref{quantum trace formula},
$$\mathrm{Tr}^{q}(X) = \prod_{i} tr([x_{i,1}] [x_{i,2}]\cdots ).$$
For any $a=(a_{i}) \in \mathrm{GL}_{\mathbf{d}}$, by 
\eqref{formula: group action on differential operators} we have
\begin{align*}
		a. \mathrm{Tr}^{q}(X) & = \prod_{i} a. tr\big([x_{i,1}] [x_{i,2}]\cdots  \big) 
		= \prod_{i} \Big(\sum_{s,k_{1},\cdots } a.[x_{i,1}]_{s,k_{1}} \cdots  a.[x_{i,j_{i}}]_{k_{j_{i}},s} \Big) .
\end{align*}
For example, 
for $a$ as above and $x \in \overline{Q}$,
\begin{equation}\label{eq:Lieact}
\begin{split}
a.[x]=\left\{
\begin{array}{ll}
a_{t(x)}^{-1} [x] a_{s(x)}, &\mbox{if}\; x \in Q_{1},\\[2mm]
a_{s(x)} ^{-1} [x] a_{t(x)}, &\mbox{if}\; x \in \overline{Q}_{1} \diagdown Q_{1}.
\end{array}
\right.
\end{split}
\end{equation}
On the right hand side of \eqref{eq:Lieact}, all products are matrix products.
Consequently, we have $a. \mathrm{Tr}^{q}(X) = \mathrm{Tr}^{q}(X)$.
\end{proof}

\begin{lemma}\label{lemma: q.trace preserves ideals}
Suppose $Q$ is a finite quiver, $\mathbf{d}$ is a dimension vector. Then there is a unique character 
$\chi_{0}$ of $\mathfrak{gl}_{\mathbf{d}}$
such that
\begin{equation}\label{eq:chi0}
\mathrm{Tr}^{q}\Big( \mathbf{N}(Q)_{\hbar}\widehat{(\mathbb{K}\overline{Q}
\sum_{a\in Q}[a,a^{*}])_{\natural}} \mathbf{N}(Q)_{\hbar} \Big) \subset 
\Big( \mathcal{D}_{\hbar}(\mathrm{Rep}(Q,\mathbf{d}))\ (\tau- \hbar \chi_{0}) 
(\mathfrak{gl}_{\mathbf{d}}) \Big)^{\mathrm{GL}_{\mathbf{d}}}.
\end{equation}
\end{lemma}
\begin{proof}

The character $\chi_{0}$ is obtained as follows. 
Since $\mathrm{Tr}^{q}$ is linear, without loss of generality, 
we choose $X=x_{1}\cdots x_{r} \sum_{a\in Q} (a a^{*} - a^{*}a) 
\in (\mathbb{K}\overline{Q}\sum_{a\in Q}[a,a^{*}])_{\natural}$ and assume $s(x_{r}) = k$. Thus,
\begin{align*}
X=x_{1}\cdots x_{r} \sum_{a\in Q} (a a^{*} - a^{*}a) &
	=\sum_{a\in Q,t(a)=k} x_{1}\cdots x_{r} a a^{*} - \sum_{a\in Q,s(a)=k} x_{1}\cdots x_{r} a^{*}a.
\end{align*}
Now, lift $X$ to $\mathbf{N}(Q)_{\hbar}$ (see the discussion
after Proposition \ref{proposition:PBW}), and suppose the lifting
\begin{displaymath}
\widehat{X}=\sum_{a\in Q,t(a)=k} \widehat{x_{1}\cdots x_{r} a a^{*}} 
- \sum_{a\in Q,s(a)=k} \widehat{x_{1}\cdots x_{r} a^{*}a}.
\end{displaymath}
Applying $\mathrm{Tr}^{q}$ to $\widehat{X}$, we get
\begin{align*}
	\mathrm{Tr}^{q}(\widehat{X})
	=&\sum_{a\in Q,t(a)=k} \sum_{l_{1},..,l_{r+2}} [x_{1}]_{l_{1},l_{2}}\cdots [x_{r}]_{l_{r},l_{r+1}} [a]_{l_{r+1},l_{r+2}} [a^{*}]_{l_{l+2},l_{1}}
	\\
	&- \sum_{a\in Q,s(a)=k} \sum_{l_{1},..,l_{r+2}} [x_{1}]_{l_{1},l_{2}}\cdots [x_{r}]_{l_{r},l_{r+1}} [a^{*}]_{l_{r+1},l_{r+2}} [a]_{l_{r+2},l_{1}}
	\\
	=&\sum_{l_{1},..,l_{r+1}}[x_{1}]_{l_{1},l_{2}}\cdots [x_{r}]_{l_{r},l_{r+1}} \sum_{a\in Q,t(a)=k} \sum_{l_{r+2}} [a]_{l_{r+1},l_{r+2}} [a^{*}]_{l_{r+2},l_{1}}
	\\
	&-\sum_{l_{1},..,l_{r+1}}[x_{1}]_{l_{1},l_{2}}\cdots [x_{r}]_{l_{r},l_{r+1}} \sum_{a\in Q,s(a)=k} \sum_{l_{r+2}}\hbar \delta^{l_{1}} _{l_{r+1}} + [a]_{l_{r+2},l_{1}} [a^*]_{l_{r+1},l_{r+2}}
	\\
	=&\sum_{l_{1},..,l_{r+1}}[x_{1}]_{l_{1},l_{2}}\cdots [x_{r}]_{l_{r},l_{r+1}} \Big(  \sum_{a\in Q,t(a)=k} \sum_{l_{r+2}} [a]_{l_{r+1},l_{r+2}} \frac{\partial}{\partial [a]_{l_1,l_{r+2}}}
	\\
	 &-\sum_{a\in Q,s(a)=k} \sum_{l_{r+2}} [a]_{l_{r+2},l_{1}} \frac{\partial}{\partial [a]_{l_{r+2},l_{r+1}}} - \sum_{a\in Q,s(a)=k} d_{t(a)} \delta^{l_{1}} _{l_{r+1}} \hbar \Big)
	\\
	=&\sum_{l_{1},..,l_{r+1}}[x_{1}]_{l_{1},l_{2}}\cdots [x_{r}]_{l_{r},l_{r+1}} \Big(-\tau (e^{k}_{l_{1},l_{r+1}}) - \sum_{a\in Q,s(a)=k}d_{t(a)}\delta^{l_{1}}_{l_{r+1}} \hbar\Big).
\end{align*}
From the above identity, we see that the character is
\begin{displaymath}
	\chi_{0}=-\sum_{k \in Q_{0}}\Big(\sum_{a\in Q,s(a)=k} d_{t(a)}\Big)tr_{k},
\end{displaymath}
where $tr_{k}$ is taking the trace of the $k$-th matrix. Also, this
$\chi_0$ is the unique character satisfying \eqref{eq:chi0}.
By $(\ref{quantum trace formula})$ and the definition of 
the product on $\mathbf{N}(Q)_{\hbar}$, 
we have that $$\mathrm{Tr}^{q} \Big(\mathbf{N}(Q)_{\hbar}\widehat{(\mathbb{K}
\overline{Q}\sum_{a\in Q}[a,a^{*}])_{\natural}} \Big)\subset 
\Big( \mathcal{D}_{\hbar}(\mathrm{Rep}(Q,\mathbf{d}))\ 
(\tau-\hbar\chi_{0}) (\mathfrak{gl}_{\mathbf{d}})\Big)^{\mathrm{GL}_{\mathbf{d}}}.$$

We next show the quantum trace of $\widehat{(\mathbb{K}\overline{Q}
\sum_{a\in Q}[a,a^{*}])_{\natural}}\mathbf{N}(Q)_{\hbar}$ also lies in 
$\Big( \mathcal{D}_{\hbar}(\mathrm{Rep}(Q,\mathbf{d}))\ (\tau-\hbar \chi_{0}) (\mathfrak{gl}_{\mathbf{d}}) \Big)^{\mathrm{GL}_{\mathbf{d}}}$.
Without loss of generality, we prove this statement 
for $\widehat{X}\&(y_{1},g_{1})\cdots (y_{s},g_{s})$ with $g_{1}<\cdots
<g_{s}$.

In fact, by $(\ref{quantum trace formula})$ we have that
\begin{align*}
&\mathrm{Tr}^{q}\Big(\widehat{X}\&(y_{1},g_{1})\cdots (y_{s},g_{s}) \Big)\\
&= \mathrm{Tr}^{q}(\widehat{X})\mathrm{Tr}^{q}([y_{1}]\cdots [y_{s}])\\
&=\Big(\sum_{l_{1},..,l_{r+1}}[x_{1}]_{l_{1},l_{2}}\cdots 
[x_{r}]_{l_{r},l_{r+1}} \Big(-\tau (e^{k}_{l_{1},l_{r+1}}) 
- \sum_{a\in Q,s(a)=k}d_{s(a)}\delta^{l_{1}}_{l_{r+1}} 
\hbar \Big) \Big) \\
&\qquad\qquad tr([y_{1}]\cdots [y_{s}])\\
&=
\sum_{l_{1},..,l_{r+1}}[x_{1}]_{l_{1},l_{2}}
\cdots [x_{r}]_{l_{r},l_{r+1}} \Big(-\tau (e^{k}_{l_{1},l_{r+1}})\, tr([y_{1}]\cdots [y_{s}]) \\
&\qquad\qquad
-\sum_{a\in Q,s(a)=k}d_{s(a)}\delta^{l_{1}}_{l_{r+1}}tr([y_{1}]\cdots [y_{s}])\hbar \Big).
\end{align*}
Since $tr([y_{1}\cdots y_{s}])$ is $\mathrm{GL}_{\mathbf{d}}$- and hence also 
$\mathfrak{gl}_{\mathbf{d}}$-invariant,  
$\tau (e^{k}_{l_{1},l_{r+1}})$ and $tr([y]_{1}\cdots[y_{s}])$ commute with each other. 
Therefore
\begin{align*}
&\mathrm{Tr}^{q} \Big(\widehat{X}\&(y_{1},g_{1})\cdots (y_{s},g_{s}) \Big) \\
&=\Big(\sum_{l_{1},..,l_{r+1}}[x_{1}]_{l_{1},l_{2}}\cdots [x_{r}]_{l_{r},l_{r+1}} 
\Big(-\tau (e^{k}_{l_{1},l_{r+1}}) - \sum_{a\in Q,s(a)=k}d_{s(a)}
\delta^{l_{1}}_{l_{r+1}} \hbar \Big) \Big) \ tr([y_{1}]\cdots [y_{s}])\\
&=\sum_{l_{1},..,l_{r+1}}[x_{1}]_{l_{1},l_{2}}\cdots [x_{r}]_{l_{r},l_{r+1}}\ 
tr([y_{1}]\cdots [y_{s}])\, 
\big(-\tau (e^{k}_{l_{1},l_{r+1}}) \big) - \sum_{a\in Q,s(a)=k}
d_{s(a)}\delta^{l_{1}}_{l_{r+1}} \hbar ).
\end{align*}	
Thus 
$\mathrm{Tr}^{q} \Big( \mathbf{N}(Q)_{\hbar}\widehat{(\mathbb{K}\overline{Q}
\sum_{a\in Q}[a,a^{*}])_{\natural}} \mathbf{N}(Q)_{\hbar} \Big) 
\subset \Big( \mathcal{D}_{\hbar}(\mathrm{Rep}(Q,\mathbf{d}))\ 
(\tau-\hbar \chi_{0}) (\mathfrak{gl}_{\mathbf{d}}) \Big)^{\mathrm{GL}_{\mathbf{d}}}$.
\end{proof}

\begin{proposition}
\label{proposition:noncom quantum moment map gives quantum moemnt map}
Let $Q$ be a finite quiver and let  $\widehat{\mathbf{w}}$ 
be the non-commutative quantum moment map. 
For a fixed dimension vector $\mathbf{d} \in \mathbb{N}^{Q_{0}}$, the map 
$\mu^{\sharp}_{\widehat{\mathbf{w}}}: \, \mathcal{U}_{\hbar}
(\mathfrak{gl}_{\mathbf{d}}) \rightarrow \mathcal{D}_{\hbar}(\mathrm{Rep}(Q,\mathbf{d}))$
sending an arbitrary  $(g_{i})\in \mathfrak{gl}_{\mathbf{d}}$ to 
$-tr\big(  [\widehat{\mathbf{w}}](g_{i})  \big) $
is a quantum moment map.
\end{proposition}

\begin{proof}
Suppose $e^{i}_{p,q}$ is the $(q,p)$-th elementary matrix in the $i$-the summand of
$\mathfrak{gl}_{\mathbf{d}}$; then
\begin{align*}
	[\widehat{\mathbf{w}}]e^{i}_{p,q}
	&=\sum_{t(a) = i} [a][a^{*}]e^{i}_{p,q} - \sum_{s(a)=i}[a^{*}] [a] e^{i}_{p,q}\\
	&=\sum_{t(a) =i} \sum_{k,l,u,v}[a]_{k,l}[a^{*}]_{u,v} e^{i}_{k,l}e^{i}_{u,v}e^{i}_{p,q} 
	- \sum_{ s(a) = i} \sum_{k,l,u,v}[a^{*}]_{k,l}[a]_{u,v}e^{i}_{k,l}e^{i}_{u,v}e^{i}_{p,q}\\
	&=\sum_{t(a)=i}\sum_{k,l} [a]_{k,l}[a^{*}]_{l,p} e^{i}_{k,q} - \sum_{s(a) = i} \sum_{k,l} [a^{*}]_{k,l} [a]_{l,p} e^{i}_{k,q}.
\end{align*}
Taking the trace on both sides of the above equality, we have
\begin{align*}
	tr([\mathbf{w}]e^{i}_{p,q})
	&=tr\Big( \sum_{t(a)=i}\sum_{k,l} [a]_{k,l}[a^{*}]_{l,p} e^{i}_{k,q} - \sum_{s(a) = i} \sum_{k,l} [a^{*}]_{k,l} [a]_{l,p} e^{i}_{k,q} \Big)\\
	&= \sum_{t(a)=i}\sum_{l} [a]_{q,l}[a^{*}]_{l,p} - \sum_{s(a) = i} \sum_{l} [a^{*}]_{q,l} [a]_{l,p}\\
	&= \sum_{t(a)= i} \sum_{l} [a]_{q,l}\frac{\partial}{\partial (a)_{p,l}} - \sum_{s(a) = i} \sum_{l} \frac{\partial }{\partial (a)_{l,q}} [a]_{l,p}\\
	&=\sum_{t(a)= i} \sum_{l} [a]_{q,l}\frac{\partial}{\partial (a)_{p,l}} - \sum_{s(a) = i} \sum_{l} \big([a]_{l,p} \frac{\partial}{\partial (a)_{l,q}} + \hbar \delta^{p}_{q} \big)\\
	&=\sum_{t(a)= i} \sum_{l} [a]_{q,l}\frac{\partial}{\partial (a)_{p,l}} - \sum_{s(a) = i} \sum_{l} \big([a]_{l,p} \frac{\partial}{\partial (a)_{l,q}} - \hbar\sum_{s(a) = i} \sum_{l} \delta^{p}_{q} \big)\\
	&=\sum_{t(a)= i} \sum_{l} [a]_{q,l}\frac{\partial}{\partial (a)_{p,l}} - \sum_{s(a) = i} \sum_{l} \big([a]_{l,p} \frac{\partial}{\partial (a)_{l,q}} - \hbar\sum_{s(a) = i} d_{t(a)}\delta^{p}_{q} \big).
\end{align*}
Therefore $\mu^{\sharp}_{\widehat{\mathbf{w}}} (e^{i}_{pq})$ is
\begin{displaymath}
\sum_{s(a) = i} \sum_{l} [a]_{l,p} \frac{\partial}{\partial (a)_{l,q}}  
- \sum_{t(a)= i} \sum_{l} [a]_{q,l}\frac{\partial}{\partial (a)_{p,l}} 
+ \hbar \sum_{s(a) = i} \sum_{l} d_{t(a)}\delta^{p}_{q},
\end{displaymath}
which exactly equals $(\tau - \hbar \chi_{0})(e^{i}_{pq})$ 
(see Lemma \ref{lemma:gl_d action on Rep}
for the definition of $\tau$
and Lemma \ref{lemma: q.trace preserves ideals} for definition of $\chi_{0}$). 
This gives the infinitesimal action of $e^{i}_{pq}$, 
and by Definition \ref{definition:quantum moment map}, 
$\mu^{\sharp}_{\widehat{\mathbf{w}}}$ 
is the quantum moment map.
\end{proof}

\subsection{Reduction commutes with the trace map}

This subsection is to show that the classical and quantum trace maps 
preserve the classical and quantum reductions respectively.

First, it is proved by Crawley-Boevey, Etingof and Ginzburg 
in \cite{CBEG2007} that the
non-commutative reduction induces the Hamiltonian reduction, 
which we state as follows. 

\begin{proposition}[Reductions commute with trace maps]
Suppose $Q$ is a finite quiver, $\mathbf{d} \in \mathbb{N}^{Q_{0}}$ is a dimension vector. 
Then we have the following commutative diagram
\begin{equation}\label{eq:reductioncommwithtrace1}\begin{split}
\xymatrixcolsep{4pc}
\xymatrix{
(\mathbb{K}\overline{Q})_{\natural} \ar[r]^-{\mathrm{Tr}} \ar@{-->}[d] 
& \mathbb{K}[T^{*}\mathrm{Rep}(Q,\mathbf{d})] \ar@{-->}[d]\\
(\Pi Q)_{\natural} \ar[r]^-{\mathrm{Tr}} & \mathbb{K}[\mathcal{M}_{\mathbf{d}}(Q)].
}\end{split}
\end{equation}
\end{proposition}

\begin{proof}See
\cite[Theoem 6.4.3]{CBEG2007} for a proof.
\end{proof}

At the quantum level, we have a similar result, which is stated as follows.

\begin{theorem}[Quantum reductions commute with 
quantum trace maps] \label{theorem:trace map preserves reduction}
Suppose $Q$ is a finite quiver, $\mathbf{d}$ is a dimension 
vector such that the moment map $\mu$ is flat. Then we have the following commutative diagram 
\begin{equation}\label{eq:reductioncommwithtrace2}\begin{split}
\xymatrix{
\mathbf{N}(Q)_{\hbar} \ar[r]^-{\mathrm{Tr}^{q}} \ar@{-->}[d]
& 	{\mathcal{D}_{\hbar}(\mathrm{Rep}(Q,\mathbf{d}))} \ar@{-->}[d]\\
\mathcal{R}_{q}(\mathbf{N}(Q)_{\hbar}, \widehat{\mathbf{w}}) 
\ar[r]^-{\mathrm{Tr}^{q}}&\displaystyle \frac{ \big(\mathcal{D}_{\hbar}
(\mathrm{Rep}(Q,\mathbf{d})) \big)^{\mathrm{GL}_{\mathbf{d}}}}
{\big( \mathcal{D}_{\hbar}(\mathrm{Rep}(Q,\mathbf{d}))\ (\tau- \hbar \chi_{0}) (\mathfrak{gl}_{\mathbf{d}}) \big)^{\mathrm{GL}_{\mathbf{d}}}}.
	}
\end{split}
\end{equation}
\end{theorem}

\begin{proof}
Due to Remark \ref{remark:general form2}, 
without loss of generality, we prove the theorem for  
elements in $\mathbf{N}(Q)_{\hbar}$ like $(x_{1},1)(x_{2},2) \cdots (x_{r},r)$.

In fact, 
$\mathrm{Tr}^{q}\Big( (x_{1},1)(x_{2},2) \cdots (x_{r},r) \Big) 
= tr[x_{1}]tr[x_{2}] \cdots tr[x_{r}]$.
Now, we denote the two projections 
$$\mathbf{N}(Q)_{\hbar} \rightarrow \mathcal{R}_{q}(\mathbf{N}(Q)_{\hbar},\widehat{\mathbf{w}})
\quad\mbox{and}
\quad(\mathcal{D}(Rep(Q,\mathbf{d})))^{\mathrm{GL}_{\mathbf{d}} } 
\rightarrow \frac{(\mathcal{D}_{\hbar}
(\mathrm{Rep}(Q,\mathbf{d})))^{\mathrm{GL}_{\mathbf{d}}}}
{ \big( \mathcal{D}_{\hbar}(\mathrm{Rep}(Q,\mathbf{d}))\ (\tau-\hbar \chi_{0}) (\mathfrak{gl}_{\mathbf{d}})\big)^{\mathrm{GL}_{\mathbf{d}} }}$$ 
by $p_{1}$ and $p_{2}$ respectively.
Then on one hand, 
\begin{align*}
	&\mathrm{Tr^{q}} \circ p_{1} \Big(  (x_{1},1)(x_{2},2) \cdots (x_{r},r) \Big)\\
	&=\mathrm{Tr^{q}} \Big((x_{1},1)(x_{2},2) \cdots (x_{r},r) + \mathbf{N}(Q)_{\hbar}\widehat{(\mathbb{K}\overline{Q}
		\sum_{a\in Q}[a,a^{*}])_{\natural}} \mathbf{N}(Q)_{\hbar}\Big)\\
	&=\Big(\prod_{i} tr[x_{i}]\Big) + \Big( \mathcal{D}_{\hbar}(\mathrm{Rep}(Q,\mathbf{d}))\ (\tau-\hbar \chi_{0}) (\mathfrak{gl}_{\mathbf{d}}) \Big)^{\mathrm{GL}_{\mathbf{d}}}.
\end{align*}
On the other hand,
\begin{align*}
	&p_{2} \circ \mathrm{Tr^{q}} \Big(  (x_{1},1)(x_{2},2) \cdots (x_{r},r) \Big)\\
	&=p_{2} \Big(\prod_{i} tr[x_{i}]\Big)\\
	&=\Big(\prod_{i} tr[x_{i}]\Big) + \Big( \mathcal{D}_{\hbar}(\mathrm{Rep}(Q,\mathbf{d}))\ (\tau-\hbar \chi_{0}) (\mathfrak{gl}_{\mathbf{d}}) \Big)^{\mathrm{GL}_{\mathbf{d}}}.
\end{align*}
These two terms are equal and thus the diagram commutes.
\end{proof}

\subsection{Quantization commutes with trace maps}
In this subsection, we show that the trace map commutes with the quantization.
This is implicit in Schedler's paper \cite{Sch2005}, which we present below
for completeness.

\begin{theorem}[Quantizations commute with trace maps: before reduction]\label{thm:[tr,quant]}
Suppose $Q$ is a finite quiver. Then
for any $x,\ y \in (\mathbb{K}\overline{Q})_{\natural}$, we have
 $$\Phi(\frac{1}{\hbar}\mathrm{Tr}^{q}(\widehat{x} \ast \widehat{y}-\widehat{y} \ast \widehat{x}))
 =\{\mathrm{Tr}(x),\mathrm{Tr}(y)\},$$
where 
$$\Phi: \displaystyle\frac{\mathcal{D}_{\hbar} \big( \mathrm{Rep}(Q,\mathbf{d}) \big)}{\hbar \mathcal{D}_{\hbar} \big( \mathrm{Rep}(Q,\mathbf{d}) \big)} \rightarrow \mathbb{K}[T^{*}\mathrm{Rep}(Q,\mathbf{d})],
[a]_{ij}\, \mapsto\, (a)_{ij},\ [a^{*}]_{ij}\, \mapsto \, (a^{*})_{ij}.$$
In other words, we have the following commutative diagram:
\begin{equation}\label{eq:quantizationcommwithtrace1}
\begin{split}
\xymatrixcolsep{4pc}
\xymatrix{
\mathbf{N}(Q)_{\hbar} \ar[r]^-{\mathrm{Tr}^{q}} & \mathcal{D}_{\hbar}(\mathrm{Rep}(Q,\mathbf{d})) \\
(\mathbb{K}\overline{Q})_{\natural} \ar[r]^-{\mathrm{Tr}} \ar@{~>}[u] & \mathbb{K}[T^{*}\mathrm{Rep}(Q,\mathbf{d})], \ar@{~>}[u]
}
\end{split}\end{equation}
which means the trace maps commute with quantization. 
\end{theorem}

\begin{proof}
For any $x=x_{1}\cdots x_{r},\ y=y_{1}\cdots y_{s} \in (\mathbb{K}\overline{Q})_{\natural}$, 
without loss of generality, we assume their liftings are $\widehat{x}=(x_{1},1)(x_{2},2)\cdots (x_{r},r),\ 
\widehat{y}=(y_{1},1)\cdots (y_{s},s)$. By \eqref{quantum trace formula}, we have
\begin{align*}
\mathrm{Tr}^{q}(\widehat{x}\widehat{y}-\widehat{y}\widehat{x})
	&=\mathrm{Tr}^{q}(\widehat{x})\mathrm{Tr}^{q}(\widehat{y})
	-\mathrm{Tr}^{q}(y)\mathrm{Tr}^{q}(x)\\
	&=\sum_{\substack{i,j_{1}\cdots \\ u,v_{1},\cdots }} 
	\Big([x_{1}]_{i,j_{1}} [x_{2}]_{j_{1},j_{2}}\cdots [x_{r}]_{j_{r-1},i} [y_{1}]_{u,v_{1}}\cdots 
	[y_{s}]_{v_{s-1},u}\\
	&\quad\quad -  [y_{1}]_{u,v_{1}}\cdots [y_{s}]_{v_{s-1},u}[x_{1}]_{i,j_{1}} 
	[x_{2}]_{j_{1},j_{2}}\cdots [x_{r}]_{j_{r-1},i} \Big). 
\end{align*}
Now, by repeatedly applying the four conditions 
in Definition \ref{def:homogenized DO},
we can switch those terms containing $x$ and those containing 
$y$ in the first summand; for example, in the first step, we plug
$[x_{r}]_{j_{r-1},i}[y_{1}]_{u,v_{1}} = [y_{1}]_{u,v_{1}} [x_{r}]_{j_{r-1},i} + \hbar\{(x_{r})_{j_{r-1},i}, (y)_{u,v_{1}}\}$
into the right hand side of the above equality, we get a switch of positions of
$[x_r]_{j_{r-1}, i}$ and $[y_1]_{u, v_1}$. Eventually, we get
\begin{align*}
	\mathrm{Tr}^{q}(\widehat{x} \ast \widehat{y}-\widehat{y} \ast \widehat{x})
	&=\hbar \sum_{\substack{i,j_{1}\cdots \\ u,v_{1},\cdots }} \Big( \Big\{ (x_{1})_{y,j_{1}},(y_{1})_{u,v_{1}}\Big\} [x_{2}]_{ j_{1},j_{2} } \cdots [x_{r}]_{j_{r-1},i} [y_{2}]_{v_{1},v_{2}} \cdots [y_{s}]_{v_{s-1}, u}\\
	&\quad \quad +\Big\{(x_{2})_{j_{1},j_{2}}, (y)_{u,v_{1}}\Big\} [x_{1}]_{ i,j_{1} } \cdots [x_{r}]_{j_{r-1},i} [y_{2}]_{v_{1},v_{2}} \cdots [y_{s}]_{v_{s-1}, u} + \cdots \Big).
\end{align*}
On the other hand, we have 
\begin{align*}
\{ \mathrm{Tr}(x),\mathrm{Tr}(y)\}
&= \Big\{\sum_{i,j_{1}\cdots }(x_{1})_{i,j_{1}}\cdots (x_{r})_{j_{r-1},i},\ \sum_{u,v_{1},\cdots } 
(y_{1})_{u,v_{1}}\cdots (y_{s})_{v_{r-1},u} \Big\} \\
&=\sum_{\substack{i,j_{1}\cdots \\ u,v_{1},\cdots }} \Big\{ (x_{1})_{i,j_{1}}
\cdots (x_{r})_{j_{r-1},i},\  (y_{1})_{u,v_{1}}\cdots (y_{s})_{v_{r-1},u} \Big\}.
\end{align*}
Therefore,
\[\Phi(\frac{1}{\hbar}\mathrm{Tr}^{q}(\widehat{x} \ast \widehat{y}-\widehat{y} \ast \widehat{x}))
=\{\mathrm{Tr}(x),\mathrm{Tr}(y)\}.\qedhere\]
\end{proof}

In a similar way, we have

\begin{theorem}[Quantizations commute with trace maps: after reduction]
Let $Q$ be a finite quiver and let $\mathbf{d}$ be a dimension vector 
such that the moment map $\mu$ is flat. Then for any $x,\ y \in (\Pi Q)_{\natural}$, 
we have
$$\Phi(\mathrm{Tr}^{q}(\widehat{x} \ast \widehat{y}-\widehat{y} \ast \widehat{x}))
	= \{\mathrm{Tr}(x),\mathrm{Tr}(y)\}.
$$ 
In other words, 
the following diagram
\begin{equation}\label{eq:quantizationcommwithtrace2}
\begin{split}
\xymatrixcolsep{4pc}
\xymatrix{
\mathcal{R}_{q}(\mathbf{N}(Q)_{\hbar},\widehat{\mathbf{w}}) 
\ar[r]^-{\mathrm{Tr}^{q}} & \displaystyle\frac{(\mathcal{D}_{\hbar}
(\mathrm{Rep}(Q,\mathbf{d})))^{\mathrm{GL}_{d}}}{( \mathcal{D}_{\hbar}
(\mathrm{Rep}(Q,\mathbf{d}))\ (\tau-\hbar \chi_{0}) (\mathfrak{gl}_{\mathbf{d}}))^{\mathrm{GL}_{d}}} \\
(\Pi Q)_{\natural} \ar[r]^-{\mathrm{Tr}} \ar@{~>}[u]& 
\mathbb{K}[\mathcal{M}_{\mathbf{d}}(Q)]\ar@{~>}[u]
}
\end{split}
\end{equation} 
commutes,
which means
the trace map commutes with the quantization
for the preprojective algebras.
\end{theorem}

\begin{proof}
Without loss of generality, we prove this theorem for arbitrary elements 
$$(x_{1},1)(x_{2},2)\cdots (x_{r},r) + L_{\hbar}$$
and
$$(y_{1},1)(y_{2},2)\cdots (y_{s},s) + L_{\hbar}$$
in $\mathcal{R}_{q}(\mathbf{N}(Q)_{\hbar}, \widehat{\mathbf{w}})$,
where $x:= x_{1}x_{2}\cdots x_{r}$, $y:= y_{1}y_{2}\cdots y_{s}$ are cyclic paths in 
$\overline Q$ and 
$L_{\hbar}:=\mathbf{N}(Q)_{\hbar}
\widehat{(\mathbb{K}\overline{Q}\mathbf{w})_{\natural}}\mathbf{N}(Q)_{\hbar}$. 
Here we denote the equivalent classes of $x, y$ in 
$(\Pi Q)_{\natural}$ by $x+ \ker pr,\ y + \ker pr$ respectively.

In fact, by Lemma \ref{lemma: q.trace preserves ideals}, 
$$\mathrm{Tr}^{q} \Big((x_{1},1)(x_{2},2)\cdots (x_{r},r) 
+ L_{\hbar} \Big) =\big\langle tr(\prod_{i=1}^{r}  [x_{i}])\big\rangle ,$$ 
where $\big\langle tr(\prod_{i=1}^{r}  [x_{i}]) \big\rangle$ is the element in 
$\displaystyle\frac{(\mathcal{D}_{\hbar}
(\mathrm{Rep}(Q,\mathbf{d})))^{\mathrm{GL}_{d}}}{( \mathcal{D}_{\hbar}
(\mathrm{Rep}(Q,\mathbf{d}))\ (\tau-\hbar \chi_{0}) (\mathfrak{gl}_{\mathbf{d}}))^{\mathrm{GL}_{d}}}$
represented by the differential operator $tr(\prod_{i=1}^{r}  [x_{i}])$. 
Similarly, we have 
$$\mathrm{Tr}^{q} \Big((y_{1},1)(y_{2},2)\cdots (y_{s},s) 
+ L_{\hbar} \Big) = \big\langle tr(\prod_{i=1}^{s}  [y_{i}]) \big\rangle.$$
By Proposition \ref{prop:image q.tr G-inv}, 
\begin{align*}
&\big\langle tr(\prod_{i=1}^{r}  [x_{i}]) \big\rangle \big\langle tr(\prod_{i=1}^{s}  
[y_{i}]) \big\rangle - \big\langle tr(\prod_{i=1}^{s}  [y_{i}]) 
\big\rangle \big\langle tr(\prod_{i=1}^{r}  [x_{i}]) \big\rangle \\
&=\big\langle tr(\prod_{i=1}^{r}  [x_{i}])tr(\prod_{i=1}^{s}  [y_{i}]) 
- tr(\prod_{i=1}^{s}  [y_{i}])tr(\prod_{i=1}^{r}  [x_{i}]) \big\rangle\\
&=\hbar \big\langle \sum_{\substack{i,j_{1}\cdots \\ u,v_{1},\cdots }} 
\Big( \Big\{ (x_{1})_{y,j_{1}},(y_{1})_{u,v_{1}}\Big\} [x_{2}]_{ j_{1},j_{2} } 
\cdots [x_{r}]_{j_{r-1},i} [y_{2}]_{v_{1},v_{2}} \cdots [y_{s}]_{v_{s-1}, u}\\
&\quad \quad +\Big\{(x_{2})_{j_{1},j_{2}}, (y)_{u,v_{1}}\Big\} [x_{1}]_{ i,j_{1} } 
\cdots [x_{r}]_{j_{r-1},i} [y_{2}]_{v_{1},v_{2}} \cdots [y_{s}]_{v_{s-1}, u} + \cdots \Big) \big\rangle\\
&=\hbar  \sum_{\substack{i,j_{1}\cdots \\ u,v_{1},\cdots }} 
\Big( \Big\{ (x_{1})_{y,j_{1}},(y_{1})_{u,v_{1}}\Big\} \big\langle 
[x_{2}]_{ j_{1},j_{2} }\big\rangle \cdots \big\langle [x_{r}]_{j_{r-1},i} 
\big\rangle \big\langle [y_{2}]_{v_{1},v_{2}} \big\rangle \cdots 
\big\langle [y_{s}]_{v_{s-1}, u}\big\rangle \\
&\quad \quad +\Big\{(x_{2})_{j_{1},j_{2}}, (y)_{u,v_{1}}\Big\} \big\langle 
[x_{1}]_{ i,j_{1} } \big\rangle \cdots \big\langle [x_{r}]_{j_{r-1},i} \big\rangle 
\big\langle [y_{2}]_{v_{1},v_{2}} \big\rangle \cdots \big\langle 
[y_{s}]_{v_{s-1}, u} \big\rangle + \cdots \Big).
\end{align*}
Moreover, since the canonical projection
$$\pi:\big(\mathbb{K}[T^{*}\mathrm{Rep}(Q,\mathbf{d})]\big)^{\mathrm{GL}_{\mathbf{d}}} 
\rightarrow \mathbb{K}[\mathcal{M}_{\mathbf{d}}(Q)]$$
preserves Poisson brackets, we have
\begin{align*}
	\Big\{ \mathrm{Tr}(x + \ker pr), \mathrm{Tr}(y + \ker pr) \Big\} 
	&= \Big\{ \mathrm{Tr} \circ pr(x),\mathrm{Tr} \circ pr(y) \Big\} \\
	&=\Big\{ \pi \circ \mathrm{Tr}(x),\pi \circ \mathrm{Tr}(y) \Big\}\\
	&=\pi \Big(\Big\{ \mathrm{Tr}(x), \mathrm{Tr}(y) \Big\} \Big).
\end{align*}
By the commutivity of quantization and reduction on $T^{*}\mathrm{Rep(Q,\mathbf{d})}$,
the theorem is now proved.
\end{proof}

\subsection{Proof of Theorem \ref{maintheorem}}

Now we are ready to prove our main theorem.

\begin{proof}[Proof of Theorem \ref{maintheorem}]
Theorem \ref{maintheorem} (1) is exactly Theorem \ref{thm:NCQR} and
Theorem \ref{maintheorem} (2) is exactly Theorem \ref{theorem:trace map preserves reduction}.
\end{proof}

\begin{proof}[Proof of the commutativity of Diagram \eqref{maincor}]
The bottom and top diagrams are given by 
\eqref{eq:reductioncommwithtrace1}
and \eqref{eq:reductioncommwithtrace2} respectively.
The front and back diagrams are given by
\eqref{eq:quantizationcommwithtrace1}
and \eqref{eq:quantizationcommwithtrace2} respectively,
and the left and right diagrams are given by
\eqref{eq:NCquantizationcommwithreduction}
and \eqref{eq:quantizationcommwithreduction} respectively.
\end{proof}

\begin{remark} 
In \cite{AKSM2002Quasi,AMM1998Lie}, Alekseev, Kosmann-Schwarzbach,
Malkin and Meinrenken introduced the notion of 
quasi-hamiltonian reduction for a symplectic space 
with a group-valued moment map. Furthermore, the 
(geometric) quantization commutes with the quasi-Hamiltonian 
reduction for these spaces (see \cite[Section 4.3]{Mei2012}). 
In \cite{Van2008Double,Van2008Quasi}, 
Van den Bergh showed there 
is a non-commutative version of quasi-hamiltonian reduction for 
bi-symplectic spaces (or more generally, for double Poisson spaces). 
It is very plausible that our result in this paper remains valid 
for quasi-hamiltonian reductions; that is, there exists a commutativity 
of the non-commutative quantization and the non-commutative 
quasi-hamiltonian reduction, at least for quiver algebras, 
which, via the trace maps, induces the commutativity 
of the quantization and the quasi-hamiltonian reduction 
on their representation spaces. As a potential application, 
we get a quantization of the character varieties 
from non-commutative geometry. 
We hope to address this problem in a future work.
\end{remark}


\end{document}